\documentclass[reqno,11pt]{amsart}
\usepackage[utf8]{inputenc}
\usepackage[dvipsnames,prologue]{xcolor} 
\usepackage{calc,graphicx,amsfonts,amsthm,amscd,amsmath,amssymb,enumerate,dsfont,mathrsfs,paralist,bbm}

\usepackage{natbib}

\usepackage{verbatim}
\usepackage{caption}
\usepackage{microtype}
\usepackage{subcaption}
\usepackage{pdfsync}
\usepackage{stmaryrd}
\usepackage{etoolbox}
\usepackage{marginnote}
\usepackage{hyperref}
\usepackage{textcase}
\usepackage[capitalize]{cleveref}

\usepackage{lmodern}
\usepackage{subcaption} 
\usepackage{chngcntr}
\usepackage{apptools}
\AtAppendix{\counterwithin{Alemma}{section}}

\usepackage{tikz}
\usepackage{floatrow}
\usepackage{xcolor}
\usetikzlibrary{arrows}
\usetikzlibrary[patterns]
\usetikzlibrary{decorations.pathreplacing}
\usetikzlibrary{patterns.meta}
\usepackage{tkz-euclide}
\definecolor{ffqqqq}{rgb}{1,0,0}
\definecolor{qqffqq}{rgb}{0,1,0}
\definecolor{ffffff}{rgb}{1,1,1}
\definecolor{ttqqqq}{rgb}{0.07,0.07,0.07}
\colorlet{ColorGray}{gray!30}
\usetikzlibrary{shapes.misc}
\usetikzlibrary{shapes.geometric}
\usetikzlibrary{arrows.meta}
\usetikzlibrary{decorations}
\usetikzlibrary{spy}

\tikzset{cross/.style={cross out, draw=black, minimum size=2*(#1-\pgflinewidth), inner sep=0pt, outer sep=0pt}, cross/.default={1pt}}

\usepackage{booktabs}
\newcommand{\bone}{{\ensuremath{\mathbf 1}} }
\newcommand{\bzero}{{\ensuremath{\mathbf 0}} }

\reversemarginpar
\newlength\fullwidth
\numberwithin{equation}{section}

\DeclareMathSymbol{\leqslant}{\mathalpha}{AMSa}{"36} 
\DeclareMathSymbol{\geqslant}{\mathalpha}{AMSa}{"3E} 
\DeclareMathSymbol{\eset}{\mathalpha}{AMSb}{"3F}     
\renewcommand{\b}{\beta}

\def\1{\ifmmode {1\hskip -3pt \rm{I}} \else {\hbox {$1\hskip -3pt \rm{I}$}}\fi}

\newcommand{\ent}{{\rm Ent} }
\newcommand{\var}{\operatorname{Var}}

\newcommand{\id}{\mathbbm{1}}

\renewcommand{\b}{\beta}
\renewcommand{\l}{\lambda}
\renewcommand{\L}{\Lambda}

\renewcommand{\l}{\lambda}
\renewcommand{\a}{\alpha}
\renewcommand{\d}{\delta}
\renewcommand{\t}{\tau}

\newcommand{\g}{\gamma}
\newcommand{\G}{\Gamma}

\newcommand{\e}{\varepsilon}

\newcommand{\gap}{{\rm gap}}

\newtheorem{theorem}{Theorem}[section]
\newtheorem*{theorem*}{Theorem}
\newtheorem{lemma}[theorem]{Lemma}

\newtheorem{proposition}[theorem]{Proposition}
\newtheorem{corollary}[theorem]{Corollary}
\newtheorem{remark}[theorem]{Remark}

\newtheorem{claim}[theorem]{Claim}
\newtheorem{definition}[theorem]{Definition}

\newtheorem*{question*}{Question}

\newtheorem*{remark*}{Remark}
\newtheorem*{idefinition*}{Definition}

\newtheorem*{pseudotheorem*}{Pseudo-Theorem}
\newtheorem{conjecture}{Conjecture}
\newtheorem{application}{Application}

\newcommand{\cB}{\ensuremath{\mathcal B}}

\newcommand{\cD}{\ensuremath{\mathcal D}}
\newcommand{\cE}{\ensuremath{\mathcal E}}
\newcommand{\cF}{\ensuremath{\mathcal F}}

\newcommand{\cH}{\ensuremath{\mathcal H}}

\newcommand{\cL}{\ensuremath{\mathcal L}}

\newcommand{\cP}{\ensuremath{\mathcal P}}

\newcommand{\bbE}{{\ensuremath{\mathbb E}} }

\newcommand{\bbN}{{\ensuremath{\mathbb N}} }

\newcommand{\bbP}{{\ensuremath{\mathbb P}} }

\newcommand{\bbR}{{\ensuremath{\mathbb R}} }

\newcommand{\bbZ}{{\ensuremath{\mathbb Z}} }
\newcommand{\Z}{{\ensuremath{\mathbb Z}} }

\let\a=\alpha \let\b=\beta   \let\d=\delta  \let\e=\varepsilon
 \let\g=\gamma       \let\l=\lambda
            
  \let\s=\sigma \let\t=\tau   
  
   \let\G=\Gamma  \let\L=\Lambda 
\let\O=\Omega      

\title[Long time behaviour of one facilitated KCM]{Long time behaviour of one facilitated kinetically constrained models: results and open problems}

\author[Fabio Martinelli]{Fabio~Martinelli}
 \address{Dipartimento di Matematica e Fisica, Universit\`a Roma Tre}\email{fabio.martinelli@uniroma3.it}
 \author[Assaf Shapira]{Assaf Shapira}\address{MAP5, Universit\'e Paris Cit\'e}
 \email{assaf.shapira@math.cnrs.fr}
 \author[Cristina Toninelli]{Cristina Toninelli}
 \address{CEREMADE, CNRS, UMR 7534, Universit\'e Paris-Dauphine, PSL University}
 \email{toninelli@ceremade.dauphine.fr}

\begin{document}

\begin{abstract}
Kinetically constrained models (KCMs) are interacting particle systems introduced in the ‘80s by physicists to have accessible stochastic models with glassy-type dynamics. The key mechanism behind the complex evolution of these otherwise simple models is the so-called dynamical facilitation, a feature embedded into the models via appropriate kinetic constraints. KCMs are reversible with respect to a Bernoulli product measure, and the analysis of their stationary evolution has witnessed significant progress in the last decade. Unfortunately, in the interesting regime when the equilibrium density of the facilitating vertices is small, many fundamental questions concerning the non-stationary evolution of even the simplest models remain unsolved. In this paper, we discuss some of these questions, along with partial new results and conjectures, for the one facilitated model and its variants, as well as for the biased annihilating branching process.    
\end{abstract}
 \maketitle
\section{Introduction}
\subsection{State of the art and some conjectures}
Fredrickson–Andersen 1-spin facilitated model (FA-1f) is an interacting
particle system that belongs to the class of kinetically constrained models (KCMs) (see \cite{HT} for a recent review).  KCMs are Glauber-type Markov processes on $\bbZ^d$ (or on suitable other graphs) informally defined as follows. Call a vertex $x$ infected if its
state is ``0'' and healthy if ``1”. Then each  vertex becomes infected at rate $q$ and healthy at rate  $1-q$ provided a certain local {\sl constraint} is satisfied. The parameter $q\in(0,1)$ is called the {\sl infection density}.

The key feature shared by all KCMs is that the local constraint at each vertex $x$ depends only on the current state of a suitable neighborhood of $x$ and not on the state of $x$ itself. A key example is the FA-1f  model, whose constraint simply checks if there is an infection among the nearest neighbors of $x$. 
It can be easily verified that KCMs are reversible with respect to the Bernoulli product measure $\pi$ with density $q$ for the infected vertices. Despite this trivial equilibrium measure, establishing the long-time behavior of KCMs poses very challenging and interesting problems. 
Indeed, though major progress has been made in the last 20 years towards a full and rigorous understanding of the large-time behavior of the stationary process (see \cite[chapters 4-6]{HT} ), robust tools to analyze FA-1f  and general KCMs out of equilibrium are still lacking and several beautiful questions remain open (see \cite[chapter 7]{HT} ). For clarity of exposition, in the sequel, unless otherwise explicitly stated, we will use the one-dimensional lattice $\mathbb {Z} $ as the underlying graph. 

The first natural question concerns the set of stationary measures for the process. It can be proved by entropy production methods (see Corollary  \ref{cor:stationary measures}) that any stationary measure is a convex combination of the completely healthy configuration and $\pi$. In higher dimensions, the same holds for translation-invariant stationary measures.

A second basic question is to determine under which conditions on $q$ and on the initial distribution $\nu$, the law  of the process at time $t$ converges to the equilibrium measure $\pi$ as $t\to \infty$. Natural choices for $\nu$ are the Bernoulli product measure $\mu_{q_0}$ with $q_0>0, q_0\neq q$ or the Dirac mass on the configuration with a single infection at the origin. For FA-$1$, it is natural to conjecture the following.
\begin{conjecture}
\label{conj:1}
For any $q\in (0,1)$ and any function $f$ depending on the state of finitely many vertices
\begin{equation}
\label{eq:conj1}
\lim_{t\to\infty}|\bbE_{\nu
}(f(\eta(t)))-\pi(f)|=0,
\end{equation}
provided that $\nu(\exists \text{ an infected vertex}) =1.$
\end{conjecture}
The next obvious question concerns the speed of relaxation to the reversible measure $\pi$ under some condition on $\nu$ stronger than just the a.s. existence of some infection. 
\begin{conjecture}
 \label{conj:2}
If there exists $\kappa>0$ such that $\nu(\text{no infected vertices in $[-\ell,\ell]$})=O(e^{-\kappa \ell})$ for all $\ell>0$ large enough then  for all $q>0$ the convergence in \eqref{eq:conj1} is exponentially fast.
\end{conjecture}
Unfortunately, robust tools to prove Conjectures \ref{conj:1} and \ref{conj:2} are not yet available, and the results are limited to $q$ larger than a certain threshold \cite{BCRT,Ertul,Mareche}. For general KCM on $\mathbb {Z} ^d$, $ d\ge 1$, there are results for high values of $q$ \cite{HTout} based on comparisons with suitable contact processes.
The only model for which the convergence to equilibrium is well understood for all $q>0$ and all $ d\ge 1$  is the East model (see \cite[Section 7.2]{HT} and references therein). 

The first and foremost reason for the  failure in proving the above conjectures for any $q>0$ is the fact that the process is not attractive (that is, the product partial order is not preserved by the semi-group of the process, see \cite[Sections II.2 and III.2]{Liggett05} for background). This is due to the fact that the presence of more infected sites may make certain constraints satisfied and therefore allow certain infected sites to become healthy. Consequently, many of the powerful techniques  (e.g.\ censoring or coupling arguments) which have been developed for the study of other Glauber dynamics (e.g.\ the contact process and stochastic Ising model, see \cite{Liggett05,Liggett99,Martinelli99}), fail here.

To make things worse, the usual Holley-Stroock strategy \cite{Holley89} to attack questions of this kind, for example, stochastic Ising models, does not apply here. Indeed, this approach utilizes the finiteness of the logarithmic Sobolev constant, which implies, in turn, the hypercontractivity of the semigroup. However, due to the presence of constraints, the logarithmic Sobolev constant is either infinite for the process on the whole lattice, or very large (i.e. $\Theta(n)$) for the chain on the finite set $\L_n=\{0,1,\dots,n\}$ with either appropriate infected boundary conditions ensuring irreducibility or conditioned on having at least one infection. In the latter setting, the mixing time of the chain scales linearly with $n$ and one expects the \emph{cutoff phenomenon} (see e.g. \cite[Chapter 18]{Peresbook}) to occur for all $q>0$. The cutoff has been proven in \cite{Ertul} for $q$ not too small.

Finally, a related question in higher dimensions is the scaling with $n$ of the mixing time of the (ergodic) chain on $\L^d_n$. The conjecture is a linear scaling for all values of $q>0$, but, once more, that has been proved only for $q$ large enough \cite{HTout}.

The goal of this review is to present some progress and past results around the above two conjectures for FA-1f  and three other closely related models on $\mathbb Z$: the 
\emph{East-polluted FA-1f }, the \emph{Weakly unoriented East process (or $\delta$-West model)}, and the 
\emph{Biased annihilating branching process (BABP)}. The paper is organized as follows. In Section \ref{sec:not-models}, we fix some notation and define the four models under consideration precisely. In Section \ref{sec:stationary measures} we present a general result on the stationary measures, while Sections 3,4, and 5 are devoted to results concerning the East polluted FA-1f  model, the $\d$-West model, and BABP. In Section 6, we present some preliminary results for FA-1f  starting from a finite number of vacancies.    

\subsection{Notation }
\label{sec:not-models}
 We will choose as a graph $\mathbb Z$ and we denote our \emph{configuration space} by $\Omega=\{0,1\}^{\mathbb Z}$.
Elements of $\Omega$ are called \emph{configurations} and denoted by Greek letters $\sigma,\eta,\omega,\dots$. For a configuration $\eta\in\Omega$ and a site $ x\in\bbZ$, $\eta_{x}$ denotes the \emph{state}  of $\eta$ at $ x$.  We say that $x$ is \emph{healthy} (or filled) if $\eta_{x}=1$ and \emph{infected} (empty) otherwise. More generally,  $\L\subset \bbZ$ is said to be infected if it contains an infected vertex and healthy otherwise. Sometimes, with an abuse of notation, we will apply the same notation to intervals $I\subset \bbR$ by identifying $I$ with $I\cap \bbZ$.

For $\eta\in\Omega$, we write $|\eta|=|\{x\in\bbZ:\eta_x=0\}|$ for the number of empty sites in $\eta$.
Given $q\in[0,1]$ we let $p=1-q$ and denote by $\pi$ the product Bernoulli($p$) measure on $\Omega$.  Given a probability measure $\nu$ on $\O,$ the mean and variance of a function $f:\Omega\to\bbR$ (if well defined) are denoted by $\nu(f)$ and $\var_\nu(f)$ or simply $\var(f)$ when $\nu=\pi$.

Sometimes we need to work on subsets $\Lambda\subset\bbZ$ of the lattice and in this case we write $\Omega_{\Lambda}$ for the corresponding configuration space $\{0,1\}^{\Lambda},$ and $\eta_\L$  for the 
\emph{restriction} of $\eta\in\Omega$ to $\Lambda$. 
Given two disjoint sets $\Lambda_1,\Lambda_2\subset \bbZ$, and $\eta^{(i)}\in \Omega_{\Lambda_i}, i=1,2,$ we write $\eta^{(1)}\cdot\eta^{(2)}\in \Omega_{\Lambda_1\cup\Lambda_2}$ for the configuration such that
\begin{equation}
\label{eq:def:concatenation}(\eta^{(1)}\cdot\eta^{(2)})_{ x}=\begin{cases}
\eta^{(1)}_{ x}& \text{if } x\in\Lambda_1,\\
\eta^{(2)}_{ x}& \text{if } x\in\Lambda_2.
\end{cases}\end{equation}
We denote the fully occupied (resp.\ empty) configurations in $\O$ by $\underline \bone$ (resp.\ $\underline \bzero$). Sometimes the same symbols are used to denote the same configurations in $\O_\L$ without the suffix $\L$ whenever it is clear from the context. Finally, we let $\eta^{ x}$ the configuration
obtained from $\eta$ by flipping its value
at $ x$,  i.e.
\begin{equation}
\label{eq:def:omega:flip}
\eta^{ x}_{ y}=
\begin{cases}
\eta_{ y} & y \neq  x,\\
1-\eta_{ x} & y= x.
\end{cases}\end{equation}

\subsection{The Markov process}
The Markov process that will be the object of this paper can be constructed via their self-adjoint Markov \emph{semigroup} $P_t:=e^{t\cL}$ on $\mathrm{L}^2(\pi)$, where the \emph{generator} $\cL$ is a non-negative self-adjoint operator with domain $\mathrm{Dom}(\mathcal L)$ that can be constructed in a  standard way (see e.g.\ \cite[Sections~I.3, IV.4]{Liggett05}) starting from its action on \emph{local functions} (i.e.\ functions depending on the occupancy variables on  a finite number sites):
\begin{equation}
\label{eq:generator}
\cL f=\sum_{ x\in \bbZ}c_{ x}\cdot \left(\pi_{ x}(f)-f\right).
\end{equation}
where each update rate $c_x(\cdot)$ is a local function specific to the model, which depends \emph{only} on the state of the neighbors of $x$ and \emph{not} on the state of $x$ itself, and which can be zero. 
Spelling out the definition of $\pi_{ x}$, the generator  can be equivalently rewritten as
\begin{equation}
\label{eq:generatorbis}
\cL f(\eta)=\sum_{ x\in \bbZ^d}c_{ x}(\eta)\left((1-\eta_{ x})(1-q)+\eta_{ x}q\right) (f(\eta^{ x})-f(\eta)).
\end{equation}
We further introduce the Dirichlet form 
$\cD:{\mbox{Dom}}(\cL)\to\mathbb R$ defined as  
 \begin{equation}\label{eq:dirichlet}
\cD(f)=- \pi(f \cdot \cL f)=\sum_{ x\in \bbZ}\pi\left(c_{ x} \var_{ x}(f)\right).
\end{equation}
Using the formulation \eqref{eq:generatorbis} it is not hard to verify that
\[
\cD(f)=\frac{1}{2}\int\pi(\mathrm{d}\eta)\sum_{ x\in \bbZ}c_{ x}(\eta)\left((1-\eta_{ x})(1-q)+\eta_{ x}q\right)\left(f(\eta^{x})-f(\eta)\right)^2.
\]
When the initial distribution at time $t=0$ is the probability measure $\nu$ on $\O$, the law and expectation of the KCM process on the Skorokhod space $D([0,\infty),\Omega)$ of c\`adl\`ag functions are denoted by $\bbP_{\nu}$ and $\bbE_{\nu}$ respectively (see \cite[Chapter III]{Billingsley99} for background). If $\nu$ is concentrated over a single configuration $\eta$ we
write simply $\bbP_\eta$ and $\bbE_\eta$ for $\bbP_\nu$ and $\bbE_\nu$, while if $\nu=\pi$, we simply write $\bbP$ and $\bbE$. We use $\eta(t)$ to denote the state of the KCM at time $t\ge 0$.
\subsection{One dimensional examples}
Here we briefly present the models that will be analyzed in this paper.
\subsubsection{The one facilitated process (FA-1f )}
For the FA-1f  process the constraint satisfies
\begin{equation}\label{cFA-1f }c_x(\eta)=
\id_{\{\eta_{x-1}+\eta_{x+1}\neq 2\}}.
\end{equation} 
Informally, FA-1f  can be described 
via its so-called \emph{graphical representation} as follows: 
each vertex ${x}\in\bbZ$ is equipped with a unit intensity Poisson process, whose atoms $t_{ x,k}$ for $k\in\bbN$ are the \emph{clock rings}. We are further given independent Bernoulli random variables $s_{x,k}$ with parameter $1-q$, called \emph{coin tosses}. At the clock ring $t_{ x,k}$, if the current configuration has at least one empty site among the nearest neighbours of $x$, we update the state of $x$ to $s_{ x,k}$. Such updates are called \emph{legal}. If, on the contrary, the above requirement or constraint is not satisfied, no update is performed, and the configuration remains unchanged.

FA-$1$f can also be defined on finite or infinite subsets $\Lambda\subset \mathbb Z$ (we write $\Lambda\Subset\bbZ$ when we assume that $\Lambda$ is finite). In this case, the most natural choice is to imagine that the configuration  is defined also outside $\Lambda$, where it is frozen and equal to some reference configuration $\sigma\in\Omega_{\bbZ\setminus\Lambda}$,  the \emph{boundary condition}. Then, for $ x\in\Lambda$, $\eta\in\Omega_{\Lambda}$, the constraint is defined as 
\begin{equation}
\label{eq:def:cx:finite}
c^{\sigma}_{ x}(\eta)= c_{ x}(\eta\cdot\sigma)
\end{equation}
(recall \eqref{eq:def:concatenation} and \eqref{cFA-1f }).  

The generator and  Dirichlet form of the process on $\Omega_{\Lambda}$ with boundary condition $\s$, denoted by $\cL^{\sigma}$ and
$\mathcal{D}^{\sigma}$ respectively, are obtained by restricting the sums in \eqref{eq:generator} and \eqref{eq:dirichlet} to sites in $\Lambda$ and substituting $c_{ x}$ with $c^{\sigma}_{ x}$. We similarly denote by $\bbP_{\nu}^{\sigma}$ and $\bbE_{\nu}^\sigma$
the law and expectation of the process with initial distribution $\nu$ and by $\eta^\sigma(t)$ the process at time $t$. Note that $\pi_\Lambda$ is reversible for this process.

\subsubsection{The biased annihilating branching process (BABP)}
The BABP constraints are now defined as
\begin{equation}\label{cBABP}c_x=2-\eta_{x-1}-\eta_{x+1}.\end{equation}
 Namely, if $x$ has two empty nearest neighbours, it is updated at rate $2$; if it has one empty nearest neighbour, it is updated at rate $1$; if it has no empty nearest neighbours, it cannot be updated. We refer the reader to \cite{SudburyNeuhauser93}.

\subsubsection{The East polluted FA-1f }\label{sec:East polluted}
The vertices of $\bbZ$ are divided  into two classes: those of type FA-1f  and those of type East. If $x$ is of type FA-1f  we write $x\sim F$, otherwise $x\sim E$.  We assume that there are infinitely many vertices of type East to the left and to the right of the origin. For example, the East sites could be distributed over $\Z$ according to a non-trivial product Bernoulli measure.  If all vertices of $\Z$ are of type East, then we call the model the \emph{East model}. 

The constraints in this case are set to 
\begin{equation}
    \label{eq:EF constraints}
    c_x(\eta)=
    \begin{cases}
       \id_{\{\eta_{x-1}+\eta_{x+1}\neq 2\}}  & \text{ if $x\sim F$}\\
       1-\eta_{x-1} & \text{ if $x\sim E$}.
    \end{cases}
\end{equation}
In other words, vertices of type FA-1f  are updated as in the FA-1f  process (cf. \eqref{cFA-1f }), while vertices of type East only query their left\footnote{In the original East model, the queried neighbor is the right one. Here we choose the left one in such a way that the growth of infection occurs from left to right.} neighbor to see if their constraint is satisfied, as in the East process.

\subsubsection{The weakly unoriented East process (or $\delta$-West process)}

Fix $\d\in [0,1]$. The constraints are  now defined as
\begin{equation}
\label{cdeltaWest}c_x=  (1-\eta_{x-1})+\delta (1-\eta_{x+1})
\end{equation}
Notice that for $\delta=0$, the $\delta$-West process is completely oriented and it becomes the East process, while for $\delta=1$, it becomes the BABP process (with no preferred direction).

\section{A general result on the stationary measures}
\label{sec:stationary measures}

\newcommand{\range}{\operatorname{r}}
\global\long\def\bp#1#2{\operatorname{BP}^{#2}(#1)}

In this part, we consider a general KCM on $\Z^d$, with constraint $c_x(\eta)$ which depends on the occupation of sites at a distance at most $\range$ from $x$. In the sequel for an integer $N$, we will write $\L_N$ for the box $\{-N,\dots,N\}^d$.
 It is not hard to verify that, since the constraint $c_{ x}(\eta)$ does not depend on $\eta_{x}$, the dynamics of all the models we defined satisfy detailed balance w.r.t.\ the product measure $\pi$. 
Therefore, $\pi$ is reversible (i.e.\ $\pi(f\cdot P_t g)=\pi(g\cdot P_t f)$ for all $f,g\in \mathrm{L}^2(\pi)$ and $t\ge0$) and therefore it is an invariant measure for the process (i.e.\ $\pi P_t=\pi$ for all $t\ge0$). However, $\pi$ is \emph{not} the unique invariant measure; indeed, the Dirac measure on the fully healthy  configuration is also invariant. 

Reference \cite{SudburyNeuhauser93} characterizes the stationary measures of BABP. Here we present a general argument which we apply, in addition to the models illustrated above, also to the East model and to the FA2f model in $\bbZ^d, d\ge 2$, for which the constraints require at least two infections among the nearest neighbors. We note that the same argument used for the East model can be applied to the North-East model, reproducing a result from \cite{KordzakhiaLalley}.

In order to understand the stationary measures of a KCM, we need two main ingredients: (i)the identification of the ergodic components; (ii) a detailed balance equation describing how probabilities change within an ergodic component.

To understand the ergodic component of a KCM, we need to introduce an intimately related deterministic process,
\emph{bootstrap percolation}.
\begin{definition}
The bootstrap percolation is a deterministic dynamics in discrete
time, where at each step, sites for which the KCM constraint is satisfied
becomes empty. This is a monotone process; hence, starting from any
configuration $\eta$, it converges to a limiting configuration that
we denote $\bp{\eta}{}$.
We say that a configuration $\eta$ is \emph{stable} for the bootstrap
percolation if $\bp{\eta}{}=\eta$, and denote the set of stable configurations
by $\mathcal{E}$. Note that the constant configurations $\underline \bzero$
and $\underline \bone$ always belong to $\mathcal{E}$.
Finally, we may consider the bootstrap percolation restricted to a
set $\Lambda\subset\Z^{d}$, with filled boundary conditions. It that
case, we denote the final configuration by $\bp{\eta}{\Lambda}$.
\end{definition}

The stable configurations for the bootstrap percolation characterize
the ergodic components of the KCM, as described in the following result \cite[Corollary 3.7]{HT}:
\begin{theorem}
Two configurations $\eta$ and $\eta'$ belong to the same ergodic
component (i.e., for any finite $\Lambda\subset\Z^{d}$ there is a
sequence of legal flips bringing $\eta$ to a configuration which
agrees with $\eta'$ on $\Lambda$) if and only if $\bp{\eta}{}=\bp{\eta'}{}$.
\end{theorem}

Inside an ergodic component, the stationary measures will be described
by detailed balance. 
\begin{definition}
We say that  $\mu$ satisfies detailed balance if   for any finite set $\L\subset \bbZ^d$ and any $ x$ such that $x+[-\range,\range]^{d}\in\Lambda$
\begin{equation}
c_{x}(\eta_{\Lambda})\pi(\eta_{\Lambda}^{x})\mu(\eta_{\Lambda})=c_{x}(\eta^x_{\Lambda})\pi(\eta_{\Lambda})\mu(\eta_{\Lambda}^{x})=c_{x}(\eta_{\Lambda})\pi(\eta_{\Lambda})\mu(\eta_{\Lambda}^{x}).\label{eq:detailed_balance}
\end{equation}   
\end{definition}

\begin{theorem}
(\cite{Holley}) Assume that $\mu$ is stationary and translation invariant.
Then, detailed balance (\ref{eq:detailed_balance}) holds.
\end{theorem}

Actually, translation invariance is not needed in one dimension. We will indeed prove the following result:
\begin{theorem}
\label{thm:1d_detailed_balance}Consider a one dimensional model,
and assume that $\mu$ is stationary. Then detailed balance (\ref{eq:detailed_balance})
holds.
\end{theorem}

We combine these two ingredients in two different mechanisms: when
the bootstrap percolation spreads out from within a box, and when
it invades a box from the outside.
\begin{definition}
Fix a configuration $\eta$, a stable configuration $\beta\in\mathcal{E}$,
and a finite box $\Lambda$. We say that $\eta$ \emph{$\beta$-spans
$\Lambda$ internally }if $\bp{\eta}{\Lambda}_{\Lambda}=\beta_{\Lambda}$,
i.e., if $\bp{\eta}{\Lambda}$ agrees with $\beta$ on $\Lambda$.
\end{definition}

\begin{definition}
Fix a configuration $\eta$, a stable configuration $\beta\in\mathcal{E}$,
and two finite boxes $\Lambda\subset\Lambda'$. We say that $\eta$
\emph{$\beta$-spans $\Lambda$ externally} from $\Lambda'$ if $\bp{\eta_{\L'\setminus \L}\cdot \mathbf 1_\L}{\Lambda'}_{\Lambda}=\beta_{\Lambda}$.
We denote the
set of configurations that $\beta$-span $\Lambda$ externally from
$\Lambda'$ by $E_{\beta}(\Lambda',\Lambda)$.
\end{definition}
Under the assumption of detailed balance \eqref{eq:detailed_balance}, we can exploit the above mechanisms to try identifying the stationary measure.
For the internal spanning one, we have:
\begin{lemma}
\label{lem:internally_spanned}Let $\mu,\mu'$ be two measures
satisfying detailed balance (\ref{eq:detailed_balance}), such
that $\mu'$ is a product measure. Assume that for some $\beta\in\mathcal{E}$,
both measures are supported on the ergodic component associated with
$\beta,$ and
that  
\begin{equation}
\lim_{N\to \infty}\mu\left(\bp{\eta}{\Lambda_{N}}=\beta_{\Lambda_{N}}\right)=\lim_{N\to \infty}\mu'\left(\bp{\eta}{\Lambda_{N}}=\beta_{\Lambda_{N}}\right)=1.
\end{equation}
Then $\mu=\mu'$.
\end{lemma}

As an immediate corollary, in the FA-1f  model, if we write a stationary
measure $\mu$ as the sum $\mu=\mu(\bp{\eta}{}=\mathbf{0})\mu\left(\cdot|\bp{\eta}{}=\mathbf{0}\right)+\mu(\bp{\eta}{}=\mathbf{1})\mu\left(\cdot|\bp{\eta}{}=\mathbf{1}\right)$,
we obtain:
\begin{corollary}
\label{cor:stationary measures}
Let $\mu$ be a stationary measure for the one-dimensional FA-1f . Then
$\mu=\lambda\pi+(1-\lambda)\delta_{\mathbf{1}}$ for some $\lambda\in[0,1]$. The same holds for the BABP and for the $\d$-West process,  $\delta\in (0,1]$.
\end{corollary}

Another consequence of Lemma \ref{lem:internally_spanned} is the
uniqueness of the invariant measure under certain hypotheses; we note
here, that in particular these are satisfied by the FA2f model (see \cite[Chapter 3]{HT}).
\begin{theorem}
\label{thm:emptiable_stationary_measure}Consider a KCM such that $\pi(\bp{\eta}{\Lambda_{N}}=\mathbf{0}_{\L_N})\xrightarrow{N\to\infty}1$.
Let $\mu$ be a measure satisfying detailed balance (\ref{eq:detailed_balance}),
and assume $\bp{\eta}{}=\mathbf{0}$ $\mu$-a.s.. Then $\mu=\pi$.
\end{theorem}

The next lemma helps describe the stationary reversible measures for external
spanning mechanism:
\begin{lemma}
\label{lem:externally_spanned}Let $\mu,\mu'$ be two measures
satisfying detailed balance (\ref{eq:detailed_balance}), such that
that $\mu'$ is a product measure. Assume that, for some $\beta\in\mathcal{E}$ and any $N$,
\begin{equation}
\lim_{L\to \infty}\mu\left(E_{\beta}(\Lambda_{L},\Lambda_{N})\right)=\lim_{L\to \infty}\mu'\left(E_{\beta}(\Lambda_{L},\Lambda_{N})\right) = 1.
\end{equation}
Then $\mu=\mu'$.
\end{lemma}

One application of this lemma is the identification of all stationary
measures of the one-dimensional East model. In this case the set $\cE$ is countable and it consists of $\mathbf 0,\mathbf 1,$ and all configurations $\beta$ for which there exists $i\in \bbZ$ such that $\beta$ is infected in $\{i,i+1,\dots\}$ and healthy elsewhere. 
\begin{theorem}
\label{thm:east_stationary_measures}Let $\mu$ be a stationary measure
for the one-dimensional East model. Then 
\begin{equation}
\mu=\sum_{i\in\Z\cup\{-\infty,\infty\}}\lambda_{i}\mu^{(i)},
\end{equation}
for some positive sequence $\left(\lambda_{i}\right)_{i\in\Z\cup\{\pm\infty\}}$,
where $\mu^{(-\infty)}=\pi, \mu^{(\infty)}=\d_{\mathbf 1},$ and for any $i\in \Z$ the measure $\mu^{(i)}$ is a product measure with
marginals
\begin{equation}
\label{eq:mui}
\mu^{(i)}(\eta_{x}=1)=\begin{cases}
1 & \text{if }x<i,\\
0 & \text{if }x=i,\\
p & \text{if }x>i.
\end{cases}
\end{equation}
\end{theorem}
The tools developed in this section are limited to either one-dimensional models or translation-invariant measures. However, special properties of the East model allow us to obtain a general result in any dimension.

For $x,y\in (\bbZ\cup \{\pm \infty\})^d$, we say that $y$ precedes $x$ if each coordinate of $y$ is not larger than the corresponding coordinate of $x$. 
We also write $\cP^{-}_x$ for the set of vertices which strictly precede $x$, and $\cP^+_x$ for the set of vertices strictly preceded by $x$.  With this notation, the constraint $c_x$ of the $d$-dimensional East model requires at least one vacancy among the neighbors of $x$ in $\cP^-_x$. In particular, under the corresponding bootstrap percolation, a single vacancy at $x$ is able to empty $\cP^+_x$. Therefore, the $0$'s of any $\b \in \cE$ can be written as $\cup_{x\in I} \cP_x^+$ for some set $I \subset (\bbZ\cup \{\pm \infty\})^d$. Taking $I_\b$ as the minimal such set, we observe that $\cE$ is in one-to-one correspondence with the collection of sets  $I \subset (\bbZ\cup \{\pm \infty\})^d$ satisfying $x,y\in I \Rightarrow y \notin \cP_x^+$. 
We note that $\mathbf 0 \in \cE$ corresponds to the singleton $I_{\mathbf 0}=\{(-\infty,\dots,-\infty)\}$, and $\mathbf 1 \in \cE$ to the singleton $I_{\mathbf 1}= \{(\infty,\dots,\infty)\}$.

For a stable configuration $\b,$ let 
 $\pi_{\beta}$ be the measure $\pi$ conditioned
on the event $\{\bp{\eta}{}=\beta\}$.
It is immediate to verify that $\pi_\b$ is the product measure with marginals   
 \begin{equation}
\label{eq:muibis}
\pi_\b(\eta_{x}=1)=\begin{cases}
1 & \text{if $x\notin \cup_{x\in I_\b}\{\cP^+_x \cup \{x\}\}$},\\
0 & \text{if $x\in I_\b$}\\
p & \text{otherwise.}
\end{cases}
\end{equation}

The next theorem generalizes Theorem \ref{thm:east_stationary_measures} in higher dimensions. Since $\mathcal{E}$ is now uncountable, the sum in Theorem \ref{thm:east_stationary_measures} should be replaced by in integral with respect to some measure $\mu_*$ on $\mathcal{E}$:
\begin{theorem}
    \label{thm:d-East} Let $\mu$ be a stationary measure for the $d$-dimensional East model, and define $\mu_*$ as the law of the random variable $\bp{\eta}{}$ when $\eta\sim \mu$. Then
    \begin{equation}
        \mu = \int_\cE d\mu_*(\b) \pi_\b.
    \end{equation}
\end{theorem}

In view of the discussion above and of Theorem \ref{thm:d-East}, we propose the following conjecture.
\begin{conjecture}\label{conj:1bis}
Let $\mu$ be a stationary measure of a KCM. Then there exists a probability
measure $\mu_{*}$ on $\mathcal{E}$ such that $\mu=\int_\cE d\mu_{*}(\beta)\pi_{\beta}$,
where $\pi_{\beta}$ is given by the measure $\pi$ conditioned
on the event $\{\bp{\eta}{}=\beta\}$.
\end{conjecture}

\subsection{Proof of Theorem \ref{thm:1d_detailed_balance}}

The proof follows the same steps as \cite[chapter IV.5]{Liggett05},
see also \cite{SudburyNeuhauser93}. We will use the same notation to
explain briefly how the argument adapts to KCMs, and where it fails
in $d=2$. To simplify the notation we assume $q\le p$. 

Recall the following definitions from \cite{Liggett05}:
\begin{align*}
\Gamma^{\mu}(x,\zeta)  = & \, \mu(c_{x}(\eta)\left(q\eta(x)+p(1-\eta(x))\right)\id_{\zeta=\eta_{\Lambda}}),& \quad x\in\Lambda\text{ and }\zeta\in\Omega_{\Lambda},\\
\alpha_{\Lambda}^{\mu}(x)  = &  \sum_{\zeta\in\Omega_{\Lambda}}\left(\Gamma^{\mu}(x,\zeta)-\Gamma^{\mu}(x,\zeta^{x})\right)\log\left(\frac{\Gamma^{\mu}(x,\zeta)}{\Gamma^{\mu}(x,\zeta^{x})}\right),& \quad  x+[-\range,\range]\subseteq\Lambda,\\
\beta_{\Lambda}^{\mu}(x)  = & \sum_{\zeta\in\Omega_{\Lambda}}\left|\Gamma^{\mu}(x,\zeta)-\Gamma^{\mu}(x,\zeta^{x})\right|,& \quad  x\in\Lambda,
\end{align*}
and the lemma:
\begin{lemma}
\label{lem:liggett_5_8}\ 
\begin{enumerate}
\item $0\le\alpha_{\Lambda}^{\mu}(x)\le\alpha_{\Lambda'}^{\mu}(x)$ for
$x\in\Lambda\subset\Lambda'$.
\item $\beta_{\Lambda}^{\mu}(x)^{2}\le2\alpha_{\Lambda}^{\mu}(x)$ for all
$x$ such that $x+\left[-\range,\range\right]\in\Lambda$.
\end{enumerate}
\end{lemma}
The proof of Theorem \ref{thm:1d_detailed_balance} is based on the
fact that the entropy production of a stationary measure is $0$.
\cite[Lemma 5.1]{Liggett05} shows that the entropy production can be
decomposed into bulk and boundary contributions $h_{\Lambda}^{\circ}(\mu)+h_{\Lambda}^{\partial}(\mu)$,
where 
\begin{eqnarray}
h_{\Lambda}^{\circ}(\mu) & = & -\frac{1}{2}\sum_{\substack{x\in\Lambda\\
x+[-\range,\range]\subseteq\Lambda
}
}\alpha_{\Lambda}^{\mu}(x),\label{eq:bulk_entropy}\\
h_{\Lambda}^{\partial}(\mu) & = & \sum_{\substack{x\in\Lambda\\
x+[-\range,\range]\not\subseteq\Lambda
}
}\sum_{\zeta\in\Omega_{\Lambda}}\left(\Gamma^{\mu}(x,\zeta)-\Gamma^{\mu}(x,\zeta^{x})\right)\log\frac{\mu(\zeta)}{\pi(\zeta)}.
\end{eqnarray}

The next ingredient we will need is to bound $h_{\Lambda}^{\partial}(\mu)$
using $\beta_{\Lambda}^{\mu}$. \cite{Liggett05} \footnote{In \cite{Liggett05} the sum defining $h_{\Lambda}^{\circ}(\mu)$ is over
all $x\in\Lambda$, with and additional term in $h_{\Lambda}^{\partial}(\mu)$
to compensate} uses the fact that $\Gamma^{\mu}(x,\zeta)$ is bounded away from
$0$; this is not the case for KCMs, due to the degeneracy of rates.
We will instead use the following lemma:
\begin{lemma}
\label{lem:boundary_entropy_upper_bound}For any $\lambda_{0}>1$,
\[
h_{\Lambda}^{\partial}(\mu)\le\frac{\lambda_{0}}{2}\sum_{\substack{x\in\Lambda\\
x+[-\range,\range]\not\subseteq\Lambda
}
}\beta_{\Lambda}^{\mu}(x)+\frac{p^{2}}{q}\times2\range\times\lambda_{0}e^{-\lambda_{0}}.
\]
\end{lemma}

\begin{proof}
By symmetrizing $\omega$ and $\omega^{x}$, we can express 
\begin{eqnarray*}
h_{\Lambda}^{\partial}(\mu) & = & \frac{1}{2}\sum_{\substack{x\in\Lambda\\
x+[-\range,\range]\not\subseteq\Lambda
}
}\sum_{\zeta\in\Omega_{\Lambda}}\left(\Gamma^{\mu}(x,\zeta)-\Gamma^{\mu}(x,\zeta^{x})\right)\lambda(x,\zeta)
\end{eqnarray*}
where $\lambda(x,\zeta)=\log\left(\frac{\mu(\zeta)}{\pi(\zeta)}\frac{\pi(\zeta^{x})}{\mu(\zeta^{x})}\right)$.

Note that, since $\frac{\pi(\zeta^{x})}{\pi(\zeta)}\le\frac{p}{q}$,
\[
\sum_{\zeta\in\Omega_{\Lambda}}\mu(\zeta^{x})\lambda(x,\zeta)\id_{\lambda(x,\zeta)>\lambda_{0}}=\sum_{\zeta\in\Omega_{\Lambda}}\mu(\zeta^{x})\frac{\mu(\zeta)}{\pi(\zeta)}\frac{\pi(\zeta^{x})}{\mu(\zeta^{x})}\, e^{-\lambda}\id_{\lambda(x,\zeta)>\lambda_{0}}\le\frac{p}{q}e^{-\lambda_{0}}
\]
and similarly
\[
\sum_{\zeta\in\Omega_{\Lambda}}\mu(\zeta)\left|\lambda(x,\zeta)\right|\id_{\lambda(x,\zeta)<-\lambda_{0}}\le\lambda_{0}e^{-\lambda_{0}}.
\]

Then 
\begin{align*}
h_{\Lambda}^{\partial}(\mu) &=  \frac{1}{2}\sum_{\substack{x\in\Lambda\\
x+[-\range,\range]\not\subseteq\Lambda
}
}\sum_{\zeta\in\Omega_{\Lambda}}\left(\Gamma^{\mu}(x,\zeta)-\Gamma^{\mu}(x,\zeta^{x})\right)\, \lambda(x,\zeta)\left(\id_{\left|\lambda(x,\zeta)\right|\le\lambda_{0}}+\id_{\lambda(x,\zeta)>\lambda_{0}}+\id_{\lambda(x,\zeta)<-\lambda_{0}}\right)\\
 &\le  \frac{1}{2}\sum_{\substack{x\in\Lambda\\
x+[-\range,\range]\not\subseteq\Lambda
}
}\sum_{\zeta\in\Omega_{\Lambda}}\left|\Gamma^{\mu}(x,\zeta)-\Gamma^{\mu}(x,\zeta^{x})\right|\, \lambda_{0} +
\frac{1}{2}\sum_{\substack{x\in\Lambda\\
x+[-\range,\range]\not\subseteq\Lambda
}
}\sum_{\zeta\in\Omega_{\Lambda}}\Gamma^{\mu}(x,\zeta)\, \lambda\id_{\lambda>\lambda_{0}} \\
&\qquad +\frac{1}{2}\sum_{\substack{x\in\Lambda\\
x+[-\range,\range]\not\subseteq\Lambda
}
}\sum_{\zeta\in\Omega_{\Lambda}}\Gamma^{\mu}(x,\zeta^{x})\ \left|\lambda\right|\id_{\lambda<-\lambda_{0}}\\
 & \le  \frac{\lambda_{0}}{2}\sum_{\substack{x\in\Lambda\\
x+[-\range,\range]\not\subseteq\Lambda
}
}\beta_{\Lambda}^{\mu}(x)+p\times\#\left\{ x\in\Lambda:x+[-\range,\range]\not\subseteq\Lambda\right\} \times\frac{p}{q}\lambda_{0}e^{-\lambda_{0}}.
\end{align*}
\end{proof}
In the following, in order to simplify notation, we will consider
boxes of the type $\Lambda=[-k\range,k\range]$, and use the subscript
$k$ rather than $\Lambda$.

We decompose the sum (\ref{eq:bulk_entropy}) into layers of width
$\range$, and by Lemma \ref{lem:liggett_5_8}:
\[
-2h_{k}^{\circ}(\mu)\le\sum_{l=1}^{k}\sum_{(l-2)\range\le|x|<(l-1)\range}\alpha_{l}^{\mu}(x).
\]
On the other hand, $-2h_{\Lambda}^{\circ}(\mu)=2h_{\Lambda}^{\partial}(\mu)\le2\range\ p(1+\ln\frac{p}{q})$,
so $\sum_{l=1}^{k}\sum_{(l-2)\range\le|x|<(l-1)\range}\alpha_{k}^{\mu}(x)$
is bounded and increasing, hence $\sum_{(l-2)\range\le|x|<(l-1)\range}\alpha_{l}^{\mu}(x)\xrightarrow{l\to\infty}0$.
By Lemma \ref{lem:liggett_5_8}:
\[
\sup_{(l-2)\range\le|x|<(l-1)\range}\beta_{l}^{\mu}(x)\xrightarrow{l\to\infty}0,
\]
and together with the convexity of the absolute value 
\[
\sup_{(k-1)\range\le|x|<k\range}\beta_{k}^{\mu}(x)\le\sup_{(k-1)\range\le|x|<k\range}\beta_{k+1}^{\mu}(x)\xrightarrow{k\to\infty}0.
\]

Finally, we combine this result with Lemma \ref{lem:boundary_entropy_upper_bound}
to conclude that, for any $\lambda_{0}>1$,
\[
\limsup_{k\to\infty}h_{k}^{\partial}\le\frac{p^{2}}{q}\times2\range\times\lambda_{0}e^{-\lambda_{0}},
\]
and taking $\lambda_{0}$ to infinity shows that $h_{k}^{\partial}\xrightarrow{k\to\infty}0$.

This is what we need in order to prove the theorem: since $\mu$ is
stationary, $h_{k}^{\partial}=\frac{1}{2}\sum_{x}\alpha_{k}^{\mu}(x)$
is positive and increasing with $k$, therefore $h_{k}^{\partial}\xrightarrow{k\to\infty}0$
implies that $\alpha_{k}^{\mu}(x)=0$ for all $k$ and $x\in[-k-\range,k+\range]$.
That is,
\[
\sum_{\zeta\in\Omega_{k}}\left(\Gamma^{\mu}(x,\zeta)-\Gamma^{\mu}(x,\zeta^{x})\right)\log\left(\frac{\Gamma^{\mu}(x,\zeta)}{\Gamma^{\mu}(x,\zeta^{x})}\right)=0,
\]
and since all terms are positive, for any $\zeta\in\Omega_{k}$
\[
c_{x}(\zeta)\left(\mu(\zeta)\left(q\zeta_{x}+p(1-\zeta_{x})\right)-\mu(\zeta^{x})\left(q(1-\zeta_{x})+p\zeta_{x}\right)\right)\log\left(\frac{\mu(\zeta)\left(q\zeta_{x}+p(1-\zeta_{x})\right)}{\mu(\zeta^{x})\left(q(1-\zeta_{x})+p\zeta_{x}\right)}\right)=0.
\]

If $\mu(\zeta)=\mu(\zeta^{x})=0$ then equation (\ref{eq:detailed_balance})
is clearly satisfied. If $\mu(\zeta)=0\neq\mu(\zeta^{x})$, then $c_{x}(\zeta)$
must be equal $0$, and again equation (\ref{eq:detailed_balance})
is satisfied. If both $\mu(\zeta)$ and $\mu(\zeta^{x})$ are non-zero,
then $c_{x}(\zeta)=0$ or $\mu(\zeta)\left(q\zeta_{x}+p(1-\zeta_{x})\right)=\mu(\zeta^{x})\left(q(1-\zeta_{x})+p\zeta_{x}\right)$,
and also in this case equation (\ref{eq:detailed_balance}) is satisfied.
\qed

\begin{remark}
The proof of Theorem \ref{thm:1d_detailed_balance} relies on the
fact that the error term $\frac{p^{2}}{q}\times2\range\times\lambda_{0}e^{-\lambda_{0}}$
in Lemma \ref{lem:boundary_entropy_upper_bound} is uniform in the
size of the box, so we are able to take $\lambda_{0}$ to infinity
after $k$. In dimension $2$, we would get a term proportional to
$k$, requiring us to take $\lambda_{0}\approx\log(k)$ in order to
make it negligible. Unfortunately, an extra factor of $\log(k)$ multiplying
$\sum_{x}\beta_{\Lambda}^{\mu}(x)$ in the bound on $h_{\Lambda}^{\partial}(\mu)$
breaks down the proof of \cite{Liggett05}.
\end{remark}

\subsection{Proof of Lemma \ref{lem:internally_spanned}}

Fix $\eta\in\Omega_{N}$ for some finite $N$. We need to show that
$\mu(\eta)=\mu'(\eta)$.

Let $L>N$, $\zeta\in\Omega_{L}$, and $\psi\in\Omega_{\partial\Lambda_{L}}$,
where $\partial\Lambda_{L}=[-L-\range,L+\range]^{d}\setminus\Lambda_{L}$.
For $x\in\Lambda_{L}$, if $c_{x}(\psi\zeta)=1$ then by detailed
balance (\ref{eq:detailed_balance})
\[
\frac{\mu(\psi\zeta^{x})}{\mu'(\psi\zeta^{x})}=\frac{\mu(\psi\zeta)}{\mu'(\psi\zeta)}.
\]
More generally, for any $\zeta'\in\Omega_{L}$ such that $\bp{\zeta}{\Lambda_{L}}=\bp{\zeta'}{\Lambda_{L}}$,
\[
\frac{\mu(\psi\zeta')}{\mu'(\psi\zeta')}=\frac{\mu(\psi\zeta)}{\mu'(\psi\zeta)}.
\]
In particular, the ratio $\frac{\mu(\psi\zeta)}{\mu'(\psi\zeta)}$
is constant for all $\zeta$ such that $\bp{\zeta}{\Lambda_{L}}=\beta_{\Lambda_{L}}$,
i.e., 
\[
\mu(\psi\zeta)=Z(\psi,\beta,L)\ \mu'(\psi\zeta)\qquad\text{if }\bp{\zeta}{\Lambda_{L}}=\beta_{\Lambda_{L}}.
\]
If we now sum over $\psi$, since $\mu'$ is a product measure, we
obtain 
\[
\mu(\zeta)=\left(\sum_{\psi}\mu'(\psi)Z(\psi,\beta,L)\right)\mu'(\zeta)=Z(\beta,L)\ \mu'(\zeta),\qquad\text{if }\bp{\zeta}{\Lambda_{L}}=\beta_{\Lambda_{L}},
\]
with $Z(\beta,L)=\frac{\mu\left(\bp{\zeta}{\Lambda_{L}}=\beta_{\Lambda_{L}}\right)}{\mu'\left(\bp{\zeta}{\Lambda_{L}}=\beta_{\Lambda_{L}}\right)}$.

We are now ready to go back to our configuration $\eta$:
\begin{eqnarray*}
\mu(\eta) & = & \sum_{\substack{\zeta\in\Omega_{L}\\
\zeta_{\Lambda_{N}}=\eta
}
}\mu(\zeta)=\sum_{\substack{\bp{\zeta}{\Lambda_{L}}=\beta_{\Lambda_{L}}}
}\mu(\zeta)+\sum_{\substack{\bp{\zeta}{\Lambda_{L}}\neq\beta_{\Lambda_{L}}}
}\mu(\zeta)\\
 & = & Z(\beta,L)\sum_{\substack{\bp{\zeta}{\Lambda_{L}}=\beta_{\Lambda_{L}}}
}\mu'(\zeta)+\sum_{\substack{\bp{\zeta}{\Lambda_{L}}\neq\beta_{\Lambda_{L}}}
}\mu(\zeta)\\
 & = & Z(\beta,L)\mu'(\eta)+\sum_{\substack{\bp{\zeta}{\Lambda_{L}}\neq\beta_{\Lambda_{L}}}
}\left(\mu(\zeta)-Z(\beta,L)\mu'(\zeta)\right).
\end{eqnarray*}
Finally, since we assume $\mu\left(\bp{\zeta}{\Lambda_{L}}=\beta_{\Lambda_{L}}\right),\mu'\left(\bp{\zeta}{\Lambda_{L}}=\beta_{\Lambda_{L}}\right)\xrightarrow{L\to\infty}1$,
the normalization constant $Z(\beta,L)\xrightarrow{L\to\infty}1$
and the second term of the sum vanishes. We thus obtain $\mu(\eta)=\mu'(\eta)$,
concluding the proof of the lemma. \qed

\subsection{Proof of Theorem \ref{thm:emptiable_stationary_measure}}

By Lemma \ref{lem:internally_spanned}, it suffices to show that,
under the assumptions of the theorem, 
\begin{equation}\label{suffices}
\mu\left(\bp{\eta}{\Lambda_{N}}=\mathbf{0}_{\Lambda_{N}}\right)\xrightarrow{N\to\infty}1.
\end{equation}
We start with the following lemma, showing that, conditioned on the
ergodic component, $\mu$ can be replaced by $\pi$:
\begin{lemma}
For any event $E$ on $\Omega_{L}$,
\begin{equation}
\mu(E)=\sum_{\beta_{L}\in\mathcal{E}_{L}}\pi\left(E|\bp{\eta}{\Lambda_{L}}=\beta_{L}\right)\mu\left(\bp{\eta}{\Lambda_{L}}=\beta_{L}\right).
\end{equation}
\end{lemma}

\begin{proof}
Since any two configurations in $\left\{ \eta:\bp{\eta}{\Lambda_{L}}=\beta_{L}\right\} $
can be connected by a sequence of legal flips, the argument in the
proof of Theorem \ref{lem:internally_spanned} shows that $\frac{\mu(\eta)}{\pi(\eta)}$
is a constant over this set. That is, $\mu(\cdot|\bp{\eta}{\Lambda_{L}}=\beta_{L})=\pi(\cdot|\bp{\eta}{\Lambda_{L}}=\beta_{L})$,
which proves the lemma.
\end{proof}
We will also need the following lemma:
\begin{lemma}
Fix $L>N$ and $\beta_{L}\in\mathcal{E}_{L}$ vanishing on $\Lambda_{N}$
(i.e., $\beta_{\Lambda_{N}}=\underline \bzero_{\Lambda_{N}}$). Then 
\begin{equation}
\pi\left(\bp{\eta}{\Lambda_{N}}_{\Lambda_{N}}=\underline \bzero_{\Lambda_{N}}|\bp{\eta}{\Lambda_{L}}=\beta_{L}\right)\ge1-o_{N}(1),
\end{equation}
where $o_{N}(1)$ stands for a function of $N$ (not depending on
$L$) that tends to $0$ as $N$ tends to infinity.
\end{lemma}

\begin{proof}
Since $\beta_{L}$ vanishes on $\Lambda_{N}$, the event $\bp{\eta}{\Lambda_{L}}=\beta_{L}$
is decreasing in the variables $\left(\eta_{x}\right)_{x\in\Lambda_{N}}$.
The event $\bp{\eta}{\Lambda_{N}}_{\Lambda_{N}}=\underline \bzero_{\Lambda_{N}}$
is also decreasing in the same variables. Therefore, by the FKG inequality
and the product structure of $\pi$,
\begin{multline}
\pi\left(\bp{\eta}{\Lambda_{N}}_{\Lambda_{N}}=\underline \bzero_{\Lambda_{N}},\bp{\eta}{\Lambda_{L}}=\beta_{L}|\left(\eta_{x}\right)_{x\notin\Lambda_{N}}\right)\ge\\
\pi\left(\bp{\eta}{\Lambda_{N}}_{\Lambda_{N}}=\underline \bzero_{\Lambda_{N}}|\left(\eta_{x}\right)_{x\notin\Lambda_{N}}\right)\pi\left(\bp{\eta}{\Lambda_{L}}=\beta_{L}|\left(\eta_{x}\right)_{x\notin\Lambda_{N}}\right),
\end{multline}
and the lemma follows since the event $\bp{\eta}{\Lambda_{N}}_{\Lambda_{N}}=\underline \bzero_{\Lambda_{N}}$
does not depend on $\left(\eta_{x}\right)_{x\notin\Lambda_{N}}$,
and its probability tends to $1$.
\end{proof}
We can now use these two lemmas to prove \eqref{suffices}. For $L>N$,

\begin{eqnarray*}
\mu\left(\bp{\eta}{\Lambda_{N}}_{\Lambda_{N}}=\underline \bzero_{\Lambda_{N}}\right) & = & \sum_{\beta_{L}\in\mathcal{E}_{L}}\pi\left(\bp{\eta}{\Lambda_{N}}_{\Lambda_{N}}=\underline \bzero_{\Lambda_{N}}|\bp{\eta}{\Lambda_{L}}=\beta_{L}\right)\mu\left(\bp{\eta}{\Lambda_{L}}=\beta_{L}\right)\\
 & \ge & \sum_{\substack{\beta_{L}\in\mathcal{E_{L}}\\
\beta_{L}=0\text{ on }\Lambda_{N}
}
}(1-o_{N}(1))\mu\left(\bp{\eta}{\Lambda_{L}}=\beta_{L}\right)+0\\
 & \ge & \left(1-o_{N}(1)\right)\mu\left(\bp{\eta}{\Lambda_{L}}_{\Lambda_{N}}=\underline \bzero_{\Lambda_{N}}\right).
\end{eqnarray*}
The last probability tends, as $L\to\infty$, to $\mu\left(\bp{\eta}{}_{\Lambda_{N}}=\underline \bzero_{\Lambda_{N}}\right)=1$
since $\bp{\eta}{}=\underline \bzero$ $\mu$-a.s., so we conclude 
\begin{equation}
\mu\left(\bp{\eta}{\Lambda_{N}}_{\Lambda_{N}}=\underline \bzero_{\Lambda_{N}}\right)\xrightarrow{N\to\infty}1.
\end{equation}
The result now follows from Lemma \ref{lem:internally_spanned}.
\qed

\subsection{Proof of Lemma \ref{lem:externally_spanned}}

As in the proof of Lemma \ref{lem:internally_spanned}, we will prove
that $\mu(\eta)=\mu'(\eta)$ for any $\eta\in\Omega_{N}$. Note that
if $\eta_{x}<\beta_{x}$ for some $x$ then $\eta\notin E_{\beta}(\Lambda_{L},\Lambda_{N})$
for any $L$, hence its probability is $0$ for both $\mu$ and $\mu'$.
We will therefore assume $\eta\ge\beta$.

Let $L>N$, $\zeta\in E_{\beta}(\Lambda_{L},\Lambda_{N})$, and $\eta'\in\Omega_{N}$
such that $\eta'\ge\beta$. With some abuse of notation, we consider
$\zeta$ as a configuration on $\Lambda_{L}\setminus\Lambda_{N}$.
Then (since $\zeta$ externally spans $\Lambda_{N}$), $\zeta\eta$
and $\zeta\eta'$ are connected, i.e. $\bp{\zeta\eta}{\Lambda_{L}}=\bp{\zeta\eta'}{\Lambda_{L}}$.

Just as in the proof of Lemma \ref{lem:internally_spanned}, this
implies 
\[
\mu(\zeta\eta)=Z(\beta,\zeta)\mu'(\eta),\qquad\zeta\in E_{\beta}(\Lambda_{L},\Lambda_{N})\text{ and }\Omega_{N}\ni\eta\ge\beta,
\]
where $Z(\beta,\zeta)=\mu(\zeta)$. Then 
\begin{eqnarray*}
\mu(\eta) & = & \sum_{\zeta\in E_{\beta}(\Lambda_{L},\Lambda_{N})}\mu(\zeta\eta)+\sum_{\zeta\notin E_{\beta}(\Lambda_{L},\Lambda_{N})}\mu(\zeta\eta)\\
 & = & \mu'(\eta)+\sum_{\zeta\notin E_{\beta}(\Lambda_{L},\Lambda_{N})}(\mu(\zeta\eta)-\mu(\zeta)\mu'(\eta)).
\end{eqnarray*}
The second term vanishes as $L\to\infty$ by assumption, completing
the proof of the lemma. \qed

\subsection{Proof of Theorem \ref{thm:east_stationary_measures}}
Note first that, for the one-dimensional East model, the configurations
stable for the bootstrap configurations are given by $\mathcal{E}=(\beta^{(i)})_{i\in\Z\cup\{\pm\infty\}}$,
where:
\begin{equation}\label{eq:beta}
\beta_{x}^{(i)}=\begin{cases}
1 & \text{if }x<i,\\
0 & \text{if }x\ge i.
\end{cases}
\end{equation}
Hence each measure $\mu^{(i)}$ is a product measure supported on the
event $\bp{\eta}{}=\beta^{(i)}$.

Let $\nu^{(i)}=\mu\left(\cdot|\bp{\eta}{}=\beta^{(i)}\right)$, so $\mu=\sum_{i\in\Z\cup\{-\infty,\infty\}}\lambda_{i}\nu^{(i)}$
where $\lambda_{i}=\mu(\bp{\eta}{}=\beta^{(i)})$. We therefore need
to show that $\nu^{(i)}=\mu^{(i)}$ for all $i$. We note at this point
that since we condition a stationary measure over an invariant event,
$\nu^{(i)}$ is a stationary measure for all $i$ and by Theorem \ref{thm:1d_detailed_balance}
it satisfies detailed balance.

For $i=\infty,$ both $\nu^{(i)}$ and $\mu^{(i)}$ are supported on
the singleton $\left\{ \beta^{(\infty})\right\} =\{\mathbf{1}\}$,
hence they are equal.

 For $i\in\Z$, we apply Theorem \ref{lem:internally_spanned}: both
$\nu^{(i)}$ and $\mu^{(i)}$ satisfy detailed balance and are supported
on $\left\{ \bp{\eta}{}=\beta^{(i)}\right\} $, $\mu^{(i)}$ is a product
measure, and whenever $N>|i|$ we have $\mu^{(i)}\left(\bp{\eta}{\Lambda_{N}}=\beta_{\Lambda_{N}}\right)=\nu^{(i)}\left(\bp{\eta}{\Lambda_{N}}=\beta_{\Lambda_{N}}\right)=1$.

Finally, for $i=-\infty$ the equality follows from Lemma \ref{lem:externally_spanned}.
Indeed, $\beta^{(-\infty)}=\mathbf{0}$, so a configuration is externally
spanning $\L_N$ from $\L_L, L>N,$ only if it contains a vacancy inside $\L_L$ and to the left of $\Lambda_{N}$:
\begin{equation}
E_{\mathbf{0}}(\Lambda_{L},\Lambda_{N})=\{\eta(x)=0\text{ for some }x\in[-L,-(N+1)]\}.
\end{equation}
Therefore, $E_{\mathbf{0}}(\Lambda_{L},\Lambda_{N})\xrightarrow{L\to\infty}\{\eta(x)=0\text{ for some }x\le -(N+1)\}\supseteq\{\bp{\eta}{}=\mathbf{0}\}.$
\qed
\subsection{Proof of Theorem \ref{thm:d-East}} The statement follows immediately from the next  result. Fix a stable configuration $\b\in \cE$ and recall the definition of the measure $\pi_\b$. 
\begin{lemma} Let $\eta$ be such that $\text{BP}(\eta)=\b$. Then $\lim_{t\to +\infty}\bbP_\eta(\eta_t\in \cdot)=\pi_\b$.   \end{lemma}
\begin{proof}
Fix $\b,\eta$ as in the lemma. If $\b=\mathbf 1,$ the result is obvious. If $\b=\mathbf 0$ there exists a sequence $\{x_n\}_{n=1}^\infty$ such that $\eta(x_n)=0$ and $\cup_n \cP_{x_n}^+=\bbZ^d$. Fix a local function $f.$ A little thought shows that the support of $f$ is contained in $\cP_{x_n}^+$ for some $n$. Using $\eta(x_n)=0,$ we can appeal to \cite{Chleboun2015} to conclude that $\lim_{t\to +\infty}\bbE_\eta\big(f(\eta)\big)= \pi(f)$. Suppose now that $\b$ is different from $\mathbf 0,\mathbf 1$ and recall the definition of $I_\b\subset (\bbZ\cup \{\pm \infty\})^d$. For simplicity we assume that $I_\b\subset \bbZ^d$, the general case being covered by the sequel combined with the argument used for $\b=\mathbf 0$. By construction $I$ is infected in $\b$, whereas $\cup_{x\in I}\cP^-_x$ is healthy. Necessarily, the same is true for $\eta$. Fix a function $f$ depending on finitely many variables in $\cup_{x\in I_\b}\cP^+_x,$ and choose  a finite subset $I$ of $I_\b$ and $M>0$ such that $\L=\cup_{x\in I}\cP^+_x \cap \{x\in \bbZ^d: \|x\|_\infty \le M\}$ contains the support of $f$. Because of the orientation of the East constraints, one immediately verifies that the marginal in $\O_\L$ of the law of the East process with initial condition $\eta$ coincides with the law of the East chain in $\O_\L$ with empty boundary condition at the vertices of $I$ and healthy elsewhere, and initial condition $\eta_\L$. Such a chain is ergodic, with reversible measure $\pi_\L,$ and therefore $\lim_{t\to +\infty}\bbE_\eta\big(f(\eta_t)\big)= \pi(f)$. 
\end{proof}


\section{Results for the East polluted FA1\(\rm f\) model}
\label{sec:poll}
In this section, we consider the East polluted FA-1f  model (see Section \ref{sec:East polluted}). Recall that we assume that there are infinitely many vertices of type East to the left and to the right of the origin. Note first that the set $\mathcal E$ of stable configurations
under the bootstrap transformation consists of the configurations $\mathbf 1, \mathbf 0,$ and of those configurations which, for a given $x\sim E$, are infected exactly at any $y\ge x$ and at $\{x-k,x-k+1,\dots, x-1\}$ if these sites form a maximal sequence of sites of type F to the left of $x$.  The same proof of Theorem \ref{thm:east_stationary_measures} gives that any stationary measure is the convex combination of $\pi$, the measure $\d_{\mathbf 1}$, and the following collection of probability measures $\hat \mu^{(i)}, i\sim E$. Let $j(i)=\max\{k\ge 0: i-1,\dots, i-k \text{ are all of type F}\}$. Then $\hat \mu^{(i)}$ is the product measure of the Dirac mass on the singleton all healthy sites to the left of $i-k$ and $\pi$ to the right of $i-k-1$ conditioned to have at least one infection in $[i-k,\dots,i]$ . 

Let $\nu$ be a probability measure on $\Omega$ and let $\nu_t$ be the law of the process at time $t>0$ when the initial distribution is $\nu$. 
\begin{theorem}
\label{thm:East polluted}
Assume that $\nu$-a.s. there exist infinitely many infected sites to the left of the origin. Then $\lim_{t\to \infty}\nu_t= \pi$.     
\end{theorem}
\begin{proof}
We begin with a result about the persistence of vacancies. W.l.o.g., we assume that the origin is of type East, and we write 
$x^{(n)}$ for the $n^{th}$-East site to the left of the origin.
\begin{lemma}
\label{lem:0bis}
Let
\[
p_n = \sup_{x^{(1)}<x\le 0}\sup_{\eta:\,\eta_x=0}\bbP_{\eta}(\eta_{y}(t)=1\ \forall \, y\in [x^{(n)},0]).
\]
Then 
$p_n\le p^n$.
\end{lemma}
\begin{proof}
Let us fix
$t>0, x\in [x^{(1)}+1,0]$ and an initial  configuration $\eta$ such that $\eta(x)=0$. Recall the graphical construction, denote the rings of the Poisson clock at a generic site $y$ by $\{t_{y,i}\}_{i=1}^\infty$, and for each such ring denote by  $\xi_{y,i}\in\{0,1\}$ the result of the corresponding coin toss. By construction the variables $\{\xi_{y,i}\}_{y\in \bbZ,i\in \bbN}$ are i.i.d. Bernoulli(p). 

Define now $\alpha^{(k)}$ as the \emph{last} legal ring before time $t$ at $x^{(k)}+1$ with an infection at $x^{(k)}$. More formally,  
\[
\alpha^{(k)} = \max\left \{i:\ t_{x^{(k)}+1,i}<t \text{ and } \eta_{x^{(k)}}(t_{x^{(k)}+1,i}) = 0 \right \},
\]
with the convention that $\max \emptyset = 0$.
We now observe that having at least one infection at $x\in (x^{(1)},0]$ at time $t=0$ and no infection at time $t$ in the interval $[x^{(n)},0]$ implies that each of the sites $x^{(1)}+1,\dots,x^{(n)}+1$ must have had at least one legal ring before $t$ occurring with an infection at their left neighbor and with the corresponding coin toss equal to one. Indeed, if for some $k$ there is no such ring, the infection cannot leave the interval $[x^{(k)},0]$. In other words, $\alpha^{(k)}\ge 1$ and $\xi_{x^{(k)}+1,\a^{(k)}}=1$ for $k=1,\dots,n$. Therefore,
\[
p_n \le \sup_{x^{(1)}<x\le 0}\sup_{\eta:\,\eta_x=0} \sum_{i_1=1}^\infty \dots \sum_{i_n=1}^{\infty}\bbE_\eta (\prod_{k=1}^n\bone_{\{\alpha^{(k)}=i_k\}}\, \xi_{x^{(k)}+1,i_k}).
\]
Since $x^{(1)}$ is an East site, the coin tosses in $[x^{(1)}+1,\infty)$ do not affect the process on $(-\infty,x^{(1)}]$. In particular $\xi_{x^{(1)}+1,i_1}$ is independent of $\{\alpha^{(k)}=i_k\}_{k=1}^n$. Therefore,
\[
\sum_{i_1=1}^\infty \dots \sum_{i_n=1}^{\infty}\bbE_\eta (\prod_{k=1}^n\bone_{\{\tau^{(k)}=i_k\}}\, \xi_{x^{(k)}+1,i_k}) \le p \times \sum_{i_2=1}^\infty \dots \sum_{i_n=1}^{\infty}\bbE_\eta( \prod_{k=2}^n\bone_{\{\tau^{(k)}=i_k\}}\, \xi_{x^{(k)}+1,i_k}),
\]
where we used $\sum_{i_1=1}^{\infty}\bone_{\alpha^{(1)}=i_1} \le 1$. Iterating the argument, we get 
\[
\sum_{i_1=1}^\infty \dots \sum_{i_n=1}^{\infty}\bbE_\eta (\prod_{k=1}^n\bone_{\{\tau^{(k)}=i_k\}}\, \xi_{x^{(k)}+1,i_k})\le p^n, 
\]
finishing the proof.
\end{proof}

Choose now a configuration $\eta$ with infinitely many infected sites to the left of the origin and let $z$ be one of them. W.l.o.g., we can assume that  the first vertex of type East greater than or equal to $z$ is the origin. Thanks to the lemma, any limit point $\mu$  of the law of $\eta_t$ as $t\to \infty$ satisfies  
\[
\mu(\eta_x=1\ \forall x\in [x^{(n)},0])\le p^n.
\]
By \cite[Theorem 1]{Mountford95} $\mu$ is stationary, hence using the arbitrariness of $z$ and $n$ together with the description of the stationary measures,  we conclude that $\mu$ coincides with $\pi$.    
\end{proof}

\section{Results for the $\d$-West process}\label{sec:delta}
The generator of the $\d$-West process can be written as $\cL=\cL^{E}+\d \cL^{W}, \d\in (0,1),$ where 
\begin{align*}
\cL^Ef(\eta) &=\sum_{ x\in \bbZ}(1-\eta_{x-1}) \left(\pi_x(f)-f\right),\\
\cL^Wf(\eta) &=\sum_{ x\in \bbZ}(1-\eta_{x+1}) \left(\pi_x(f)-f\right).
\end{align*}
As for the FA-1f  process, any stationary measure is the convex combination of $\d_{\mathbf 1}$ and $\pi$.  Our main result reads as follows. 
\begin{theorem}
    \label{thm:delta West}
For any $q\in (0,1)$  there exists $0<\d_0<1$ such that the following holds. Let $\nu$ be the initial law of the process, which is either a non-trivial product Bernoulli measure or, for some constant $M>0$, it is supported on configurations $\eta$ with at least one infection every $M$ vertices.  Then, for any local function   
$
\lim_{t\to \infty}\bbE_\nu(f(\eta(t))=\pi(f) 
$
and the convergence is exponentially fast.
\end{theorem}
\begin{proof}
The proof follows the general scheme adopted in \cite[Section 3]{BCRT} for the FA-1f  process with $q > 1/2$, which we summarize briefly here into three main steps.  In the sequel, for simplicity, we assume that the local function is simply $f(\eta)=\eta_0$ and that $\nu$ is concentrated on a single configuration $\eta$ that satisfies the theorem's requirement.     

{\it Step 1}. Using finite speed of propagation $\bbE_\eta(f(\eta_t))$ is approximated by the same average for the finite $\d$-West chain in the interval $\L_t=[-2t,2t]\cap \bbZ$ with infected boundary conditions at $\pm (2t+1)$ and initial condition the restriction of $\eta$ to $\L_t$.  

{\it Step 2}. Given a small positive constant $\e$, $\L_t$ is partitioned into consecutive intervals $\L_i$, each one with $\e t$ vertices. One then defines $\hat \O_{\L_t}$ as the set of those configurations in $\{0,1\}^{\L_t}$ that have at least one infection in each interval $\L_i,$ and one compares the average of $f(\eta_t)$ for the $\d$-West chain in $\L_t$ to the same average w.r.t. $\d$-West chain in $\L_t$ \emph{restricted} to $\hat \O_{\L_t}$. The assumption on $\eta$ ensures that, for $t$ large enough, each interval $\L_i$ is infected at time $t=0$. Hence, using the key Corollary \ref{cor:persistence}, for $\d$ small enough and some constant $c>0$,  with probability $1-O(e^{-c t})$ the unrestricted chain will coincide with the restricted one up to time $t$.

{\it Step 3}. In the last step, for $\e$ small enough, one successfully compares the average of $f(\eta_t)$ for the restricted chain to its equilibrium value using bounds on the logarithmic Sobolev constant of the restricted chain. Here, the key input (cf. Lemma \ref{lem:logsob}) is that the growth of the  logarithmic Sobolev constant of the restricted chain as $t\to \infty$ is $O(\e t),$ to be compared with the $\Theta(t)$ growth for the \emph{unrestricted } chain.
\vskip 0.3cm 
\noindent
We now turn to the analysis of the two key inputs that are needed for the above approach to be successful:  (i) persistence of initial infections and (ii) a bound on the logarithmic Sobolev constant of the restricted chain. 
For (i), it is convenient to consider the following graphical construction of the $\d$-West process. 
    
    Each vertex carries two independent Poisson clocks, an East clock and a West clock, of rate one and $\d$ respectively. As usual, independence is assumed through the lattice $\Z$, and when the East clock at $x$ rings, the site $x-1$ is queried. Iff $x-1$ is infected, then the ring is called \emph{legal} and the state of $x$ is resampled from the Bernoulli(p) measure. Similarly, for the West clock rings, but in this case, the queried site is $x+1$.
In the sequel we will denote by $\{t^E_{x,i}\}_{i=1}^\infty,\, \{t^W_{x,i}\}_{i=1}^\infty$ the rings of the East and West clock at $x$ respectively, and by $\{\xi^E_{x,i}\}_{i=1}^\infty,\, \{\xi^W_{x,i}\}_{i=1}^\infty$ the corresponding Bernoulli(p)-variables. \begin{definition}\ 
\label{def:1}
\begin{enumerate}[(a)]
    \item Consider the space-time $S=\bbZ \times [0,+\infty)$ and fix an integer $\ell$. A broken line $\G^+$ in $\bbR\times [0,+\infty)$ connecting  $\bbZ\times 0$ to $\bbZ\times t$ by going only North or East (up or right) and with corners, if any, occurring only at the vertices of $\bbZ\times \ell \bbN$ is called a \emph{NE-barrier}. Similarly, we define a \emph{NW-gate} $\G^-$. We write  $\G^\pm(0)\in \bbZ$ for the starting point of $\G^\pm$ and  $\G^\pm(s), s>0, s\notin \ell\bbN, $ for the spatial coordinate of the point $\G^\pm\cap \{\bbZ\times s\}\in S$. If $s$ is a multiple of $\ell$ we let $\G^\pm(s):=\lim_{s'\to s^-}\G^\pm(s)$. 
    \item Given two NE-barrier  $\G^+_1,\G^+_2$ we say that $\G^+_2$ is to the right of $\G^+_1$ if $\G^+_2(s)>\G^+_1(s)$ for any $s\le t$. Similarly for the NW-gates. 
    \item Finally, a NE-barrier or a NW-gate is good if there is no West clock ring along its vertical parts. 
\end{enumerate}
\end{definition}
\begin{remark}
\label{rem:def}\  
\begin{enumerate}[(a)]
    \item For a.a. realizations of the Poisson clocks in the interval $[0,t]$, there exist infinitely many  good NE-barriers and good NW-gates connecting $\bbZ\times 0$ to $\bbZ\times t$. Moreover, NE-barriers and NW-gates can be horizontal only at times multiple of $\ell$. Hence, a.s. any ring of the Poisson clocks along them can only occur inside their open vertical segments. 
    \item Suppose that $I=i\times (k_1\ell,k_2\ell),$ with $k_1<k_2\in \bbN$, is a (open) vertical stretch of a NE-barrier or of a NW-gate between two consecutive turns. In order that there is no West clock ring along $I$ it is sufficient that all points $\{(i,j), j=k_1,\dots,k_2-1\}$ are good, where  $(i,j)\in \bbZ\times \bbN$ is \emph{good} iff the West clock at $i$ does not ring during the time interval $[j\ell,(j+1)\ell)$. Clearly 
the events $\{(i,j)\text{ is good})\}_{(i,j)\in \bbZ\times \bbN}$ are independent and  $\bbP((i,j)\text{ is good})=e^{-\d \ell}$.    
\end{enumerate}
\end{remark}
The main properties of good NE-barriers and NW-gates are summarized in the next claim.
\begin{claim}\ 
\label{claim:1} 
\begin{enumerate}[(1)]
    \item Suppose that we have a good NE-barrier $\G^+$ starting at time $t=0$ at some positive site and a good NW-gate $\G^-$ starting at some negative site. If at time $t=0$ the origin is infected and at time $t\notin \ell \bbN$ there is no infection in $\{\G^-(t)+1,\dots, \G^+(t)\}$ then there exists $s<t, s\notin \ell \bbN,$ such that: 
\begin{enumerate}[(i)]
    \item $s$ is a East-ring for $\G^-(s)+1$ with coin toss variable equal to one;
    \item the sites $\G^-(s)+2,\dots ,\G^+(s)$ are healthy while the sites $\G^-(s)$ and $\G^-(s)+1$ are infected.
\end{enumerate}
\item Let $\G^-$ be a good NW-gate starting at $t=0$ and assume to know all the clock rings together with their corresponding coin tosses in the  region $\cup_{0\le s\le t}\, \left[\{x\le \G^-(s)\}\times s\right]$. Then, given $\{\eta_x(0)\}_{x\le \G^-(0)},$ the variables $\cup_{s\le t} \cup_{x\le \G^-(s)}\{\eta_x(s)\}$ are known.   
\end{enumerate}
\end{claim}
\begin{proof}[Proof of Claim \ref{claim:1}]\ 
\begin{enumerate}
    \item For simplicity write $\L(s)=[\G^-(s)+1,\G^+(s)]\cap \bbZ$ and observe that $\G^+(s)>0$ and $\G^-(s)<0$ for all $s\ge 0$. 
Let $
\t= \inf\{s\in [0,t): \eta(s)\restriction_{\L(s)}=1\}.
$
By construction, at time $\t$ there is a unique vacancy in $\L(\t)$ at either $\G^+(\t)$ or $\G^-(\t)+1$  which disappears because of a legal ring at its location. We can exclude the first option because a ring at $\G^+(\t)$ can only be an East ring, which, in order to be legal, needs a vacancy at $\G^+(\t)-1\in \L(\t)$.
\item It follows immediately from the fact that there are no West rings on $\G^-$ and no rings at times multiple of $\ell$.   
\end{enumerate}
\end{proof}
Write now $\G^+$ for the leftmost good NE-barrier starting to the right of the origin, $\G^-_1$ for the rightmost good NW-gate starting to the left of the origin, $\G^-_2$ for the rightmost good  NW-gate to the left of $\G^-_1$, etc. For any $s\in [0,t]$ also let  $\L_n(s)=\{\G_n^-(s)+1,\dots,\G^+(s)\}$.
\begin{lemma}
    \label{lem:2}
Fix $\eta$ such that $\eta_0=0$ and call $\cH_n(t)$ the event that $\L_n(t)$ is healthy. Then $\bbP_\eta(\cH_n(t))\le p^n$.
\end{lemma}
\begin{proof}[Proof of the lemma]
In analogy with the proof of Lemma \ref{lem:0bis}, we set 
\[
T_k=\max\{s\le t:\  \eta_{\G^-_k(s)}(s)=0\text{ and }s=t_{\G^-_k(s)+1,i}\text{ for some $i$}\},
\]
and call $\g_k=\G^-_k(T_k)$ and $\a_k$ the label $i$ of the clock ring at $\g_k+1$ such that $T_k=t_{\g_k+1,i}$. 
The same arguments used to prove part 1 of the previous claim, show that if $\L_1(0)$ is infected while $\L_n(t)$ is healthy, then necessarily $\{T_k,\g_k,\a_k\}_{k=1}^n$ are well defined and the corresponding coin tosses $\{\xi_{\g_k+1,\a_k}\}_{k=1}^n$ are all equal to one. 

Let $\cF_t$ be the $\sigma$-algebra generated by all the Poisson clocks up to time $t$. By construction, a.s.
\[
\bbP_\eta(\cH_n(t)|\cF_t)\le \sum_{y_n<\dots< y_1=-1}^\infty\dots \sum_{i_1,\dots, i_n=1}^\infty\bbE_\eta(\prod_{k=1}^n\bone_{\{\g_k=y_k\}}\bone_{\{\a_k=i_k\}}\xi_{y_k+1,i_k}|\cF_t).
\]
Since there is no West ring along $\G^-_1$, the coin tosses associated to any clock ring in the space-time region $\cup_{s=0}^t [[\G^-_1(s)+1),\infty)\times s]$ do not affect the evolution in the region $\cup_{s\le t} [(-\infty,\G^-_1(s)]\times s.$ Hence the variable $\xi_{y_1+1,i_1}$ is independent of the variables $\{\g_k,\a_k\}_{k=1}^n$ and
\begin{align*}
    &\sum_{y_n<\dots< y_1=-1}^\infty\dots \sum_{i_1,\dots, i_n=1}^\infty\bbE_\eta(\prod_{k=1}^n\bone_{\{\g_k=y_k\}}\bone_{\{\a_k=i_k\}}\xi_{y_k+1,i_k}|\cF_t)\\
    \le &\sum_{y_n<\dots< y_2=-1}^\infty\dots \sum_{i_2,\dots, i_n=1}^\infty\bbE_\eta(\prod_{k=2}^n\bone_{\{\g_k=y_k\}}\bone_{\{\a_k=i_k\}}\xi_{y_k+1,i_k}|\cF_t)\times p.
\end{align*}
By iterating the above argument, we conclude that, a.s., $\bbP(\cH_n(t)|\cF(t))\le p^n$. 
\end{proof}
We now exploit the lemma to show that for $\d$ small enough, an initial infection at the origin is likely to survive not too far from the origin. 
\begin{lemma}
 \label{lem:3}
For any $\epsilon>0$ there exists $\d_0\in (0,1)$ and $c>0$ such that for any $0<\d\le \d_0$ and any $t>0$
\[
\sup_{\eta:\,\eta_0=0}\bbP_\eta([-\epsilon t,\e t] \text{ is healthy at time $t$ })\le e^{-c\e t}.
\]
\end{lemma}
\begin{proof}[Proof of the lemma]
Fix $\kappa>0$ in such a way that oriented (NE or NW) site percolation in $\bbZ\times \bbN$ (cf. e.g. \cite{Durrett}) with parameter $e^{-\kappa}$ is supercritical. In particular (see e.g. \cite[Proposition 3.1]{KobAndersen}), for any $n$ large enough with probability greater than $1-e^{-\Theta(n)}$ there exist $\Theta(n)$ disjoint occupied NE-oriented paths connecting $\bbZ\times 0$ to $\bbZ\times n$ and  all contained in $\cup_{j=0}^n \{j,j+1,\dots,j+n\}$. 

Recall now part (b) of Remark \ref{rem:def}. If we choose the free parameter $\ell$ of Definition \ref{def:1} equal to $\lfloor \kappa/\d\rfloor$, we get that the good sites of $\bbZ\times \ell\bbN\subset S$ perform a supercritical oriented percolation with slope $\Theta(\kappa/\d)$. Given an oriented (NE or NW) good path $\g$ in  $\bbZ\times \ell\bbN\subset S$ connecting $\bbZ\times 0$ to $\bbZ\times N$, we can form a good (NE-barrier or NW-gate) by joining one vertex of $\g$ to the next one in increasing order w.r.t. the time coordinate. Notice that in this construction of a good NE-barrier or of a good NW-gate, also the points lying on a horizontal part of the line are good.  

In conclusion, if we choose $\d<O(\kappa \e)$ we get that there exist positive constants $\lambda$ and $c$ depending only on $\kappa$ such that with probability greater that $1-e^{-c\e t}$ there exists at least $\l \e t$ good NW-gates in the space-time window $[-\e t,0]\times [0,t]$ and one good NE-barrier in the space-time window $[0,\e t]\times [0,t]$ joining $\bbZ\times 0$ to $\bbZ\times t$. The proof is finished using Lemma \ref{lem:2} with $n=\lceil \l\e t\rceil$.   
\end{proof}
\begin{corollary}
\label{cor:persistence}
    In the same assumption of Lemma \ref{lem:3} there exists $c'>0$ such that for any $t>0$  
 \[
\sup_{\eta:\,\eta_0=0}\bbP_\eta(\exists s\le t: \, [-\epsilon t,\e t] \text{ is healthy at time $s$ })\le e^{-c'et}.
\]   
\end{corollary}
\begin{proof}[Proof of the corollary]
It follows immediately from a union bound over the possible Poisson clocks rings in $[-\e t,\e t]\times [0,t],$ standard large deviations bounds for the Poisson process, and the argument given in the proof of Lemma \ref{lem:3}.       
\end{proof}
We now turn to the second key input behind the approach described at the beginning of the proof.

Fix $\ell\in \bbN$ and let $\L_N$ be the union of the first $N$  consecutive discrete intervals $\L_i=\{i(\ell+1)-\ell, i(\ell+1)\}$. Let  $\hat\O_{N,\ell}$ consists of those configurations of $ \{0,1\}^{\L_N}$ that have at least one infection in each interval $\L_i, i=1,\dots,N$. We then consider the East+$\d$-West chain in $\L_N$  restricted to $\hat\O_{N,\ell}$ and with infected boundary conditions at the origin and at $x=N(L\ell+1)+1$. This restricted  chain is ergodic and reversible w.r.t. $\hat \pi_{\L_N}(\cdot)=\otimes_i \hat \pi_{\L_i},\ \hat\pi_{\L_i}(\cdot):=\pi_{\L_i}(\cdot |\L_i \text{ infected})$, with a positive spectral gap and a finite logarithmic Sobolev constant (see \cite{Diaconis} and \cite[Section 3.1]{BCRT}). The key point is that the logarithmic Sobolev constant is $O(\ell)$. More precisely,
\begin{lemma}
\label{lem:logsob}
There exists a constant $c=c(q)$ such that for any $\d\ge 0$ and any $N$, the logarithmic Sobolev constant of the restricted chain is bounded from above by $c\ell$.   \end{lemma}
\begin{proof}[Proof of the lemma]
We first recall the following result \cite[Proposition 3.10]{Martinelli99}. For any discrete probability measure $\mu$ with  finite $\mu^*:= \min_\eta \mu(\eta)$ it holds that, for any $f$, 
\begin{align}
\label{eq:aux}   {\rm Ent}_\mu(f^2)&:= \mu(f^2\log(f^2))- \mu(f^2)\log(\mu(f^2)))\nonumber \\
   &\le 2(4+2|\log(\mu^*)|)\var_\mu(f).
\end{align}
Using the fact that $\hat \pi_{\L_N}= \otimes_i\hat \pi_{\L_i}$ 
it  follows immediately from \eqref{eq:aux} that for any $f:\O_{N,\ell}\mapsto \bbR$
\[
{\rm Ent}_{\hat \pi_{\L_N}}(f^2)\le c_q \ell \sum_i \hat \pi_{\L_N}\big(\var_{\hat \pi_{\L_i}}(f)\big).
\]
For each interval $\L_i, \; i\ge 2,$ the previous block $\L_{i-1}$ has a leftmost infection. The required infection for the interval $\L_1$ is simply represented by the infected boundary condition at the origin.  We can then use the standard enlargement trick \cite[Section 4.2.2.2 equation 4.13]{HT} to extend $\var_{\hat \pi_{\L_i}}(f)$ to a larger interval whose left boundary borders the leftmost infection in $\L_{i-1}$. For the enlarged variance we can now use the Poincar\'e inequality and the spectral gap bound for the East model (i.e. with $\d=0$) to conclude that the r.h.s above is bounded from above by $c_q \times  \ell \times \gap_{\rm East}^{-1}\times \hat \cD_{\L_N}(f)$, where $\hat \cD_{\L_N}(f)$ is the restricted chain Dirichlet form of $f$.  
\end{proof}
The proof of the theorem is complete. 
\end{proof}

\section{Results for BABP}
\label{sec:BABP}
All the results concerning BABP are based on a remarkable self-duality property and on its quasi-duality with the \emph{Double Flip Process (DFP)} \cite{Loyd-Sudbury97}, an interacting particle system in which the state of each \emph{edge of $\bbZ$}, independently across $\bbZ$, is updated according to the rule: 
\begin{align*}
(0,0)\rightarrow (1,1) & \text{ with rate } (\sqrt{1+\lambda}+1)^2/2, \\
(1,1)\rightarrow (0,0) & \text{ with rate }  (\sqrt{1+\lambda}-1)^2/2,  \\
(1,0)\leftrightarrow (0,1) & \text{ with rate }  \lambda/2,  
\end{align*}
where $\l=q/p$. 
\begin{remark}
\label{rem:parity}
If one considers DFP on $\{1,\dots,n\}$ it has two ergodic components because the parity of the number of infections is preserved.    
\end{remark}
It is easy to check that DFP is reversible w.r.t. product Bernoulli measure $\hat \pi$ of parameter $\hat p=\frac{\sqrt{1+\l}+1}{2\sqrt{1+\l}},$ with $\hat q:=1-\hat p\simeq q/4$ as $q\to 0$. To avoid confusion, we will use $\bbE_\eta(\cdot)$ and $\hat \bbE_\eta(\cdot)$ in the sequel to denote the averages with respect to the BABP and the DFP, respectively, with initial state $\eta$. 
Our main result for DFP is the following.
\begin{theorem}
    \label{thm:DFP}
   For any $\l>0$, there exists  $m>0$, and for any function $f$ depending on the state of finitely many vertices, there exists a constant $C_f>0$ such that
\begin{equation}
\label{eq:DFP0}
\sup_\eta |\hat \bbE_\eta(f(\eta(t)))-\hat\pi(f)|\le C_fe^{-mt}. 
\end{equation} 
\end{theorem}
Before proving the theorem,  we first present the self-duality of BABP and its quasi-duality with DFP, and show how to combine them to derive two interesting consequences for BABP. We emphasize that the main ideas presented here originate from \cite{Sudbury99,Loyd-Sudbury97}, and that our primary new contribution is the proof of exponential ergodicity for DFP in $\bbZ$.   

Call $y=\sqrt{\lambda+1}$, and fix three subsets $B,B',D$ of $\bbZ$. We consider three independent processes $\{B(t)\}_{t\ge 0}, \{B'(t)\}_{t\ge 0},$ and $\{D(t)\}_{t\ge 0}$ in $\bbZ$ defined as follows. The set $B(t)$ represents the set of infections at time $t$ of BABP with $B(0)=B$. Similarly for $B'(t)$ with $B$ replaced by $B'$. The set $D(t)$ represents instead the set of infections at time $t$ of DFP with  $D(t)=D$. The initial sets $ B,B',D$ could be random according to a certain joint law, and the symbol $\bbE(\cdot)$ will denote the global average over the three processes including their initial randomness.  

If $|B|$ denotes the cardinality of the set $B$,  then
the \emph{self-duality} of BABP and the \emph{quasi-duality} between BABP and DFP read as follows. 
\begin{align}
 \label{eq-duality0}
  \bbE\big((-\frac{1}{\lambda})^{|B'(t)\cap B|}\big) &= \bbE\big((-\frac{1}{\lambda})^{|B(t)\cap B'|}\big), \\
    \label{eq-quasiduality}
\bbE\Big(\big(\frac{1}{y+1}\big)^{|D\cap B(t)|}\big(\frac{-1}{y-1}\big)^{|D^c\cap B(t)|}\Big)&= \bbE\Big(  \big(\frac{1}{y+1}\big)^{|B\cap D(t)|}\big(\frac{-1}{y-1}\big)^{|B\cap D^c(t)|}\Big),
\end{align}
whenever the l.h.s and the r.h.s. make sense. 
\begin{remark}
\label{rem:conv}It is not difficult to verify that, if the r.h.s. of \eqref{eq-duality0} tends to zero as $t\to \infty$ for any finite set $B',$ then the law of $ B(t)$ tends to $\pi$ as $t\to \infty$.  
\end{remark} 
\begin{application}[Linear growth of $B(t)$ starting from finitely many infections]
  Take for simplicity $B=\{0\}$ (any finite set would do as well). Finite speed of propagation implies that a.s. $|B(t)|\le O(t)$ as $t\to \infty$. 
If $D=\bbZ,$ then \eqref{eq-quasiduality} together with Theorem \ref{thm:DFP} implies that
    \begin{align*}
        \bbE\big(\big(\frac{1}{y+1}\big)^{ |B(t)|}\big)&= \big(\frac{1}{y+1}+\frac{1}{y-1}\big)\bbP(D(t)\ni \{0\}) -\frac{1}{y-1}\\
        &=\frac{1}{y-1}\Big(\frac{2y}{y+1}\bbP(D(t)\ni \{0\})-1\Big)= O(e^{-m t}).
    \end{align*}
   In particular, 
    \[
    \bbP(|B(t)| \le \delta mt)\le O(e^{-m t/2}),
    \]
   for some $\delta\in (0,1)$ small enough.  The a.s. linear growth in $t$ of $|B(t)|$ as $t\to \infty$ follows easily.    
\begin{remark}
In \cite[Corollary 2.3]{Mountford93}, it was proved that if $\l>1/3,$ then the BABP process starting from finitely many infections converges to $\pi$. The condition $\l> 1/3$ was later improved to $\l>0.0347$ in \cite[Theorem 7]{Sudbury99}.  
\end{remark}  
\end{application}
\begin{application}[Convergence of BABP to $\pi$ when the initial measure is a Bernoulli product measure]  
Let $B$ be a finite set, let $\eta'$ be distributed according to $\nu_\alpha$, the Bernoulli product measure with parameter $\alpha,$ and let $B'$ be the set of infections of $\eta'$. We distinguish between two cases for $\alpha$.
\begin{enumerate}
    \item $0<1-\alpha< \frac{2\lambda}{\lambda +1}$. In this case self-duality \eqref{eq-duality0} gives
\begin{equation}
\label{eq:dual1}
\bbE\big((-\frac{1}{\lambda})^{|B'(t)\cap B|}\big)= \bbE\big((\alpha -(1-\alpha)/\lambda)^{|B(t)|}\big).    
\end{equation}
Notice that $|\alpha-(1-\alpha)/\lambda|<1$ so that 
a.s. the r.h.s. of \eqref{eq:dual1} decreases exponentially fast in $t$ because $|B(t)|$ grows a.s. linearly in $t$. 
Using Remark \ref{rem:conv}, we conclude that the law of BABP with initial measure $\nu_{\a}$ converges to $\pi$ exponentially fast. 
\item  $\frac{2\lambda}{\lambda +1}\le 1-\alpha\le 1$. 
Let $\beta=\frac 12\big(1+\a y  \big).$ Then $\beta \in [\frac 12,1]$ and  
$
\alpha -(1-\alpha)/\lambda=\frac{\beta}{y+1}-\frac{1-\beta}{y-1}.
$
Thus, if $D\sim \nu_\beta,$ the r.h.s. of \eqref{eq-duality0} coincides with the l.h.s. of \eqref{eq-quasiduality}. On the other hand, the r.h.s. of \eqref{eq-quasiduality} is $O(e^{-mt})$ for some constant $m>0$ because of Theorem \ref{thm:DFP}. In conclusion, the l.h.s. of \eqref{eq-duality0} is  exponentially small in $t$ and the conclusion is as in the previous case.
\end{enumerate} 
\end{application}
\begin{remark}\ 
The fact that BABP started from $\nu_\alpha, \ \alpha\le 1/y,$ converges exponentially fast to $\pi$ could also be derived from the quasi-thinning relation between BABP and DFP \cite{Loyd-Sudbury97}. In this case, the law of BABP at time $t$ can be mapped into the law of DFP at time $t$ with initial distribution $\nu_{1-\beta}$. The same conclusion holds if we start BABP from an inhomogeneous Bernoulli product measure $\otimes_x \nu_{\alpha_x}$ with $\alpha_x\le 1/y$ for all $x$.  
\end{remark}
\subsection{Proof of Theorem \ref{thm:DFP}}
To fix ideas one can imagine that the local function $f$ is just the state of the origin, $\eta_0$. If $\L=\L_t=[-M t,Mt],$ finite speed of propagation implies that we can choose $M=M(\l)>0$ such that, uniformly in $\eta$, 
\[
\hat\bbE_\eta(f(\eta(t)))= \hat \bbE_{\eta^\L}(f(\eta_\L(t)))
+O(e^{-t}), 
\]
where $\eta^\L(t)$ is the DFP continuous time chain in $\L$. In the sequel, for lightness of notation, we will drop the superscript $\L$ from the notation, and we will write $\hat P_t$ for the semigroup of the DFP in $\L$. 
We now analyze $\hat \bbE_{\eta^\L}(f(\eta_\L(t))\equiv (\hat P_tf)(\eta)$. 

As explained in Remark \ref{rem:parity}, the parity $\cP(\eta)\in \{+,-\}$  of a configuration $\eta$, i.e., whether it has an even or odd number of infections, is preserved by the DFP chain. In particular, if $\hat\pi^\star$ denote the product Bernoulli($\hat p$) measure in $\L$ conditioned to $\cP=\star$, then $(\hat P_t f)(\eta)$ will converge to $\hat \pi^\star$ as $t\to \infty$ for all $\eta$ with $\cP(\eta)=\star$. It is easy to verify that there exists $c=c(\hat p)>0$ such that for all $t>0$ 
\[
|\hat\pi^+(f)-\hat\pi^-(f)|\le O(e^{-c t}),
\]
and therefore the theorem follows if we can prove that there exists $m>0$ such that, for all $t$ large enough,
\begin{equation}
    \label{eq:DFP1}
    \max_{\star}\max_{\eta\sim \star} |(\hat P_t f)(\eta)-\hat\pi^\star(f)|\le e^{-mt}.
\end{equation}
The key to proving \eqref{eq:DFP1} is the following result, whose proof is postponed to the end of the section. 

Consider the DFP in the interval $I_n=\{1,\dots,n\}$ restricted to configurations with $\cP=\star$ and with reversible measure  $\hat \pi^\star_n,$ the product Bernoulli$(\hat p)$ measure in $I_n$ conditioned to $\cP=\star$. Write $c_{sob}(n,\star)$ for its logarithmic Sobolev constant, i.e. the best constant $c$ in the Logarithmic Sobolev inequality
\[
\ent_{\hat \pi^\star_n}(f^2)\le c \cD^\star_n(f,f),\quad \forall f:\O_{I_n}^\star\mapsto \bbR,
\]
where $ \ent_{\hat \pi^\star_n}(f^2)= \hat\pi_n^\star(f^2\log(\frac{f^2}{\hat \pi^\star_n(f^2)}))$,  $\O_{I_n}^\star=\{\eta\in \O_{I_n}:\ \cP(\eta)=\star\},$ and $\cD^\star_n(f,f)$ is the Dirichlet form of the chain.
\begin{proposition}
\label{prop:sob}
$c_{sob}:= \sup_{n,\star} c_{sob}(n,\star)<+\infty.$
\end{proposition}
Back to the proof of \eqref{eq:DFP1}, let  $g_{t-s}=\hat P_{t-s}f -\hat\pi^\star(f),\, s=\log(t)^2$. 
Then, for any $\eta$ with $\cP(\eta)=\star$ and $a=1+e^{2s/c_{sob}}$, 
\begin{align*}
    |(\hat P_t f)(\eta)-\hat\pi^\star(f) | &= |(\hat P_s g_{t-s})(\eta)|\\
    \le \Big(\frac{\hat\pi^*\big(|\hat P_s g_{t-s}|^a\big)}{\hat\pi^*(\eta)}\Big)^{1/a} &\le \Big(\frac{1}{\hat\pi^*(\eta)}\Big)^{1/a}\hat \pi^*\big(g^2_{t-s}\big)^{1/2},
\end{align*}
where we used hypercontractivity of the chain in the last inequality \cite{Diaconis}. The choice of $s,a$ guarantees that $\Big(\frac{1}{\hat\pi^*(\eta)}\Big)^{1/a}=O(1)$ as $t\to \infty$, while the boundedness of $c_{sob}$ and the fact that $\hat\pi^\star(g_{t-s})=0$ imply that  
\[
\hat \pi^*\big(g^2_{t-s}\big)^{1/2}\le e^{-(t-s)/c_{sob}}\var_{\hat\pi^\star}(f)^{1/2},
\]
and \eqref{eq:DFP1} follows.
\begin{proof}[Proof of Proposition \ref{prop:sob}]
Let $\g_k= \max_{n \le 2^{k}+\sqrt{2^k}}\max_{\star}c_{sob}(n,\star)$. We will prove recursively that for all $k$ large enough 
\begin{equation}
\label{eq:DFP2}
\g_{k+1}\le (1+ \e_k)\g_k,\qquad \e_k= O(e^{-c' k}),  
  \end{equation}
  for some $c'>0$. 
Clearly \eqref{eq:DFP2} shows that $\sup_k\g_k<+\infty$, i.e. $c_{sob}<+\infty$. 

Fix an integer $n\in (2^{k}+\sqrt{2^k},2^{k+1}+\sqrt{2^{k+1}}]$ and cover the interval $I_n$ with  $\L_1=\{1,\dots, n_1\}$ and $\L_2=\{n_2,\dots, n\}$, where $n_1=\lfloor \frac n2 +n^{1/3}\rfloor$ and $n_2= \lfloor \frac n2\rfloor$. Finally, denote by  $\cF_1,\cF_2$ the $\s$-algebras generated by the variables $\{\eta_x\}_{x\in \{1,\dots, n_2-1\}}$ and $\{\eta_x\}_{x\in \{n_1+1,\dots, n\}}$ respectively. 
\begin{claim}
There exists a constant $c>0$ such that, for any function $g$ measurable w.r.t $\cF_1$ and any parity $\star$, 
\begin{equation}
    \label{eq:DFP3}
    \|\hat\pi_n^\star(g |\cF_2)-\hat\pi_n^\star(g)\|_\infty \le e^{-cn^{1/3}}\,\hat\pi_n^\star(|g|).
\end{equation}
\end{claim}
\begin{proof}[Proof of the claim]
The conditional expectation $\hat\pi_n^\star(g |\cF_2)$ depends on the variables generating $\cF_2$ only through their parity. Moreover, $g$ depends only on the first $n_2-1$ variables in $\L_1$. Hence,  \eqref{eq:DFP3} follows immediately by using the product structure of $\hat \pi$ together with the well known computation $\bbP(\text{Bin}(m,\hat p) \text{ is even })=\frac 12 +\frac 12 (1-2\hat p)^m$.     
\end{proof}
If we combine \eqref{eq:DFP3} together with \cite[Proposition 2.1]{Cesi}, we get the quasi-factorization of $\ent_{\hat \pi^\star_n}(f^2),$ namely
\begin{equation}
    \label{eq:DFP4}
   \ent_{\hat \pi^\star_n}(f^2)\le (1+O(\d_n))\hat\pi_n^\star\Big(\ent_{\hat \pi^\star_n(\cdot|\cF_2)}(f^2)+ \ent_{\hat \pi^\star_n(\cdot|\cF_1)}(f^2)\Big),
\end{equation}
where $\d_n=e^{-c\, n ^{1/3}}$. 

In order to bound the r.h.s. above, we observe that, for all $k$ large enough, the cardinality of $\L_1,\L_2$ is smaller than $2^k+\sqrt{2^k}$. Hence, the logarithmic Sobolev constant of the DFP in $\L_i,i=1,2$ with arbitrary parity is bounded from above by $\g_k$. By applying the logarithmic Sobolev inequality to each term $\hat\pi_n^\star\Big(\ent_{\hat \pi^\star_n(\cdot|\cF_i)}(f^2)\Big),i=1,2,$  we conclude that 
 \begin{equation}
 \label{eq:DFP5}
      \hat\pi_n^\star\Big(\ent_{\hat \pi^\star_n(\cdot|\cF_2)}(f^2)+ \ent_{\hat \pi^\star_n(\cdot|\cF_1)}(f^2)\Big)\le \g_k \cD^\star_n(f,f) + \g_k\cD_{\L_1\cap \L_2}^\star(f,f),
 \end{equation}
 where $\cD_{\L_1\cap \L_2}^\star(f,f)$ is the contribution to the total Dirichlet form $\cD_n^\star(f,f)$ of the updates of nearest neighbor pairs contained in $\L_1\cap \L_2$. Hence,
 \begin{equation}
     \label{eq:DFP6}
   \ent_{\hat \pi^\star_n}(f^2)\le (1+O(\d_n))\g_k\Big(\cD^\star_n(f,f) + \cD_{\L_1\cap \L_2}^\star(f,f)\Big)   
 \end{equation}
  The trick to overcome the double counting in the r.h.s. above is to spread it out over many similar choices of the sets $\L_1,\L_2$ (see \cite[Proof of Theorem 4.5]{Martinelli99}). More precisely, let $N= \lfloor n^{\frac 17}\rfloor,$ and let $\L_1^{(j)}=\{1,\dots, n_1^{(j)}\}, \L_2^{(j)}=\{n_2^{(j)},\dots,n\},j=1,\dots,N,$, where $n_1^{(j)}= \lfloor \frac n2 + j n^{1/3}\rfloor$ and $n_2^{(j)} = n_1^{(j-1)}+1$. For any $k$ large enough, the cardinality of each of the above sets is still at most $2^k+\sqrt{2^k}$ and $\L_1^{(j)}\cap \L_2^{(j)}$ are disjoint sets as $j$ varies.   We can then  apply \eqref{eq:DFP6} to each pair $\L_1^{(j)},\L_2^{(j)}$ and average the resulting inequality over $j$ to get
 \begin{align}
     \label{eq:DFP7}
     \ent_{\hat \pi^\star_n}(f^2)&\le (1+O(\d_n))\g_k\Big(\cD_n^\star(f,f)+ \frac{1}{N}\sum_{j=1}^N \cD_{\L^{(j)}_1\cap \L^{(j)}_2}^\star(f,f)\Big)\\
     &\le (1+O(\d_n))(1+ \frac{1}{N})\g_k \, \cD_n^\star(f,f).
 \end{align}
 In conclusion, for any $2^k+\sqrt{2^k} <n\le 2^{k+1}+\sqrt{2^{k+1}},$
 \[
 c_{sob}(n,\star)\le (1+O(\d_n))(1+ \frac{1}{\lfloor n^{\frac 17}\rfloor})\g_k \le (1+ O(e^{-c' k}))\g_k
 \]
for a suitable $c'>0$.  \eqref{eq:DFP2} is now established. 
\end{proof}
\subsubsection{Extension to higher dimensions}
We claim that the two main results presented in Applications 1 and 2 hold in any dimension. 

The self-duality of BABP and its quasi-duality with DFP also hold for the processes defined in $\mathbb {Z} ^d$, $ d\ge 2$. Theorem \ref{thm:DFP} also extends to higher dimensions because the logarithmic Sobolev constant of DFP in dimension $d$ can be bounded from above by the same constant in dimension $d=1$.  
To see that, suppose w.l.o.g.  $d=2$ and delete from $\bbZ^2$ all the vertical edges. On the new graph, consider the auxiliary process defined as the product of one-dimensional DFP, one for every horizontal copy of $\bbZ$. The auxiliary process is still reversible w.r.t. $\hat \pi$, and its Dirichlet form is smaller than the full DFP Dirichlet form because the contribution coming from the vertical edges is missing. Moreover, the logarithmic Sobolev constant of the auxiliary process is equal to that of the one-dimensional DFP because of the well-known tensorization property of the logarithmic Sobolev inequality \cite{Diaconis}.

\section{FA1\({\rm f}\) with finitely many infections}
In this section, we present some simple results for the FA-1f  process, starting with a finite number of infections. 

Given $\eta\in \O$ let $X^+(\eta)=\sup\{x:\ \eta_x=0\}, X^-(\eta)=\inf\{x:\ \eta_x=0\}$  be the position of the rightmost/leftmost infection of $\eta$ and set \[
Y(\eta)= \max(|X^+(\eta)|,|X^-(\eta)|),\quad  D(\eta)= X^+(\eta)-X^-(\eta).
\]
Let $\eta(t)$ be the FA-1f  process starting from $\eta$ and write $X^\pm(t),Y(t),D(t)$ for the quantities $X^\pm(\eta(t)),Y(\eta(t)), D(\eta(t))$. 
Before stating the main result on the long-time behaviour of $Y(t),D(t)$ when they are both finite, let us make some easy observations.
\begin{claim}   
\begin{enumerate}[(i)]
\item $ \limsup_{t\to +\infty} X^{+}(t)$ and $\liminf_{t\to +\infty} X^-(t)$ belong to $\{-\infty,+\infty\}$ a.s.
\item
If $\eta$ has a single vacancy at the origin then $\bbP_\eta(\limsup_{t\to +\infty} X^+(t)=+\infty)\ge 1/2$.
\item  If $Y(\eta)<+\infty$ then $\bbP_\eta(\limsup_{t\to +\infty} X^+(t)=+\infty)>0$.
\end{enumerate}
\end{claim}
\begin{proof}\ 
\begin{enumerate}[(i)]
    \item Suppose by contradiction that $\limsup_{t\to +\infty} X^{+}(t)$ is finite, say equal to $x$ for some $x\in \bbZ.$ Then, every time that $X^+(t)=x$ there is a positive chance to advance to $x+1$ contradicting the hypothesis.
\item Let $\beta=\bbP_\eta(\limsup_{t\to +\infty} X^+(t)=-\infty)$. Then  $\bbP_\eta(\liminf_{t\to +\infty}X^-(t)=-\infty)\ge \b$  and symmetry implies that $1-\b=\bbP_\eta(\limsup_{t\to +\infty} X^+(t)=+\infty)\ge \b$, i.e. $\b\le 1/2.$ \item If $\eta$ has finitely many infections, with positive probability the process can reach the configuration with a single infection at time $t=1$. Hence, the statement follows from part (ii).   \end{enumerate}  
 \end{proof}
\begin{theorem}\label{thm:frontFAfullline}
For any $q\in (0,1]$ there exist two positive constants $b,c$ such that, for any $\eta$ with finitely many infections,  a.s. as $t\to +\infty$ 
\begin{align}
\label{eq:lingr Y}   b^{-1}t\le Y(t)&\le ct\\
 \label{eq:lingr D}     c^{-1}t \le D(t)&\le ct.
\end{align}
\end{theorem}
\begin{remark}
    Notice that for BABP, the result follows at once from the linear growth in time of the cardinality of the infection proved in the first application. 
\end{remark}

\begin{proof}[Proof of \eqref{eq:lingr Y}]
We introduce the hitting times $\t_n=\inf\{t: Y(t)=n\}$ and the return times $\s_n=\inf\{t\ge \t_n: Y(t)=0\}$. In the sequel, we say that the process has an $n$-excursion if $\s_n<\t_{n+1}$ and we write $p_n$ for its probability. The duration of an $n$-excursion is $\sigma_n$. 
\begin{lemma}
\label{lem:3bis}
 There exist positive constants $\kappa_-, \kappa_+, c$ such that, for any large enough $n,$ $\bbP_\eta\big(\t_n \notin (\kappa_- n,\kappa_+ n) \big)\le e^{-n}$ and $p_n \le e^{-c' n}.$
 \end{lemma}   
\begin{proof}[Proof of the lemma]
For simplicity and without loss of generality, we assume that $\eta$ has a single infection at the origin.  

The fact that $\bbP_\eta(\t_n\le n/3)\le \frac 12 e^{-n}$ for any $n$ large enough follows from the usual finite speed of propagation argument (see e.g. \cite{HT}). Hence, we can choose $\kappa_-=1/3$. 

Next, we observe that, by construction, $\t_n$ coincides with the same hitting time for the finite FA-1f  \emph{chain} on the interval $\L_n=[-n,n]$ with healthy boundary condition at $\pm(n+1).$  For the latter, by applying  a general result on hitting times for Markov chains \cite[Proposition 3.21]{Aldous-Fill} we get
\begin{align*}
    \bbP_\eta(\t_n> t) &\le \frac{1}{p^{2n}q}\bbP_{\pi_{\L_n}}(\t_n\ge t)\le  \frac{1}{p^{2n}q}e^{- q \g_n  t},
    \end{align*}
where $\g_n\ge \g_\infty>0$ is the spectral gap of the chain on $\L_n$ with healthy boundary conditions. Hence, we can choose $\kappa_+$ large enough such that for all $n$ large enough $\bbP_\eta(\t_n> \kappa_+\,n)\le \frac 12 e^{-n}$.

In order to bound from above the probability of a $n$-excursion we observe that
the event $\{\t_{n+1} \ge \kappa_+ (n+1)\}$ is implied by the event that the process has $m=\lceil \frac{\kappa_+ + 1}{\kappa_-}\rceil$ $n$-excursions and each one of them has duration at least $\kappa_- n$. Call $\cB$ the latter event and observe that at the end of an $n$-excursion the process starts afresh from $\eta$. In particular, the probability of $m$ $n$-excursions is $p_n^m$. 
Using the first part of the lemma, we get
\begin{align*}
e^{-(n+1)}&\ge \bbP_\eta(\t_{n+1}\ge \kappa_+ (n+1))\ge \bbP_\eta(\cB)\ge p_n^m  - m e^{-n},
\end{align*}
where $m e^{-n}$ is a union bound of the probability that one $n$-excursion among the $m$ ones has a duration less than $k_- n$. In conclusion
\[
p_n\le \Big(e^{-(n+1)}+me^{-n}\Big)^{1/m}.
\]
\end{proof}
Next, we show that, once $Y(t)$ reaches level $n$, it is unlikely that it will retrace to level $\e n$ before reaching  level $n+1$. 
 \begin{lemma}
\label{lem:no retrace}
There exists $0<\e_0 <1$ and $c>0$ such that for all $0<\e\le \e_0$ and all $n$ large enough the following holds. Let $\tau_{\e n}:=\inf\{t\geq \tau_n: Y(t)=\lfloor \e n\rfloor\}$. 
$$\bbP_\eta(\tau_{\e n}<\t_{n+1})\le e^{-cn}.$$
\end{lemma}
\begin{proof}
Fix an arbitrary configuration  $\xi$ in the interval $I_{\e,n}=[-\lfloor \e n\rfloor,\lfloor \e n\rfloor]$ with at least one infection at the end points of $I_{\e,n}$ and let
$p_{\xi,\epsilon,n}(t) $ be the probability that the FA-1f  chain in $I_{\e,n}$, with healthy boundary conditions at $\pm (\lfloor \e n\rfloor+1)$ and initial state $\xi,$ has a single infection at the origin at time $t$. Using that the mixing time of the chain is $O(\e n)$, we can choose $\l=\l(q)$ large enough such that the variation distance of the law of the chain at time $t=\l \e n$ from $\pi_{I_{\e,n}}$ is smaller than $\frac 12 qp^{2 \e n}$. Hence, 
\[
\min_\xi p_{\xi,\epsilon n}(\l \e n)\ge \frac 12 qp^{2 \e n}.
\] 
Back to the FA-1f  process on the whole lattice, suppose now that $\tau_{\e n}< \t_{n+1}$ and that the clocks at $\pm (\lfloor \e n\rfloor+1)$ never ring between $\tau_{\e n}$ and $\tau_{\e n}+\l\e n $. Then the previous computation together with the strong Markov property prove that the process, after the hitting time $\t_{\e n}$ has probability at least $\frac 12 qp^{2 \e n}e^{-2 \l \e n}$ to go back to the configuration with a single infection at the origin before $\t_{n+1}$. Using Lemma \ref{lem:3}, we conclude that for all $n$ large enough
\[
 \frac 12 qp^{2 \e n}  e^{-2\l \e n}\times \bbP_\eta(\tau_{\e n}<\t_{n+1})\le p_n\le e^{-c' n},
\]
and the statement follows by taking $\e$ small enough.
\end{proof}
A simple Borel-Cantelli argument together with Lemma \ref{lem:3bis} and \ref{lem:no retrace} now proves the a.s. linear growth \eqref{eq:lingr Y} of $Y(t)$.  
\end{proof}
\begin{proof}[Proof of \eqref{eq:lingr D}]
To prove the linear growth of $D(t):=D(\eta(t))$, we proceed as follows. Let $\t_n=\inf\{t:D(t)=n\}$ and let 
$S_n$ be the discrete circle with $2n$ points which we label from $0$ to $2n-1$ in clockwise order starting e.g., from the North Pole. On $S_{n}$, consider the FA-1f  chain $\{\hat\eta(t)\}_{t\ge 0}$ starting with exactly one vacancy at the origin, and let $\hat D(t)$ be the largest distance on $S_n$ between two vacancies of $\hat \eta(t)$.   Let also $\hat \t_n=\inf\{t>0:\ \hat D(t)=n\}$. 
\begin{claim}
$\bbP(\t_n>t)\le \bbP(\hat \t_n>t)$.    
\end{claim}
\begin{proof}[Proof of the claim]
If $\t_n>t$ then for any $s<t$ we can map the process on $\bbZ$, $\eta(s),$ into a legal configuration $\phi(\eta(s))$ of the FA-1f  chain on $S_n$ as follows. Suppose that $X^+(\eta(s))= k\times 2n + m,$ with $m=[0,2n-1]$ and $k\in \bbZ$. Then we set $\phi(\eta(s))_{m-j}=\eta(s)_{k\times 2n +m-j}, j=0,\dots, n,$ and $\phi(\eta(s))_x=1$ everywhere else, where $m-j\in S_n$ is the vertex reached in $j$ counterclockwise steps from $m$. Starting with a single vacancy at the origin of $\bbZ$, the evolution of $\phi(\eta(s))$ is the correct one for the FA-1f  chain on $S_n$ up to time $\t_n$.    
\end{proof}
Using the claim and the fact that the FA-1f  chain on $S_n$ restricted to the set of configurations with at least one vacancy has a spectral gap that is uniformly positive in $n$, we obtain the analog of Lemma \ref{lem:3bis} for $\t_n$.  
Next, we define the $n$-excursion for the process $D(\cdot)$. 

As before, we let $\s_n=\inf\{t>\t_n:\ D(t)=0\},$ we say that there is a $n$-excursion if $\s_n<\t_{n+1},$ and we write $p_n$ for the probability of the latter event. Notice that if the process on $\bbZ$ had $k$ $n$-excursions before $\t_{n+1}$ then so does the chain on the torus $S_{n+2}$. Moreover, at the end of each excursion, both processes have a single vacancy somewhere, and they start afresh. Since each excursion lasts $\O(n)$ w.h.p. because of the finite speed of propagation, we easily get the analog of Lemma \ref{lem:3bis} and  
\ref{lem:no retrace} and from there the linear growth of $D(t)$.    
\end{proof}

\section*{Acknowledgements}
We would like to thank G. Ascione, A. Di Crescenzo and D. Marinucci for their kind invitation to contribute to the special number of the Bollettino dell’Unione Matematica Italiana dedicated to Probability Theory and Mathematical Statistics.

\bibliographystyle{plain}
\bibliography{FA1f}

\end{document}